\documentclass{amsart}

\usepackage{mathrsfs}
\usepackage{epsfig,graphicx,color,subfigure}
\usepackage{latexsym}                 
\usepackage{amsmath,amsfonts,xkeyval,bm,upgreek,amscd,wasysym,amssymb,textcomp,mathdots,accents}
\usepackage[english]{babel}
\usepackage{listings}
\usepackage{multicol}
\usepackage{tikz-cd}
\usepackage{booktabs}
\usepackage{paralist}
\usepackage{hyperref}
\usepackage{parcolumns}
\usepackage{enumitem}
\usepackage{mathtools}
\usepackage[numbers,sort&compress]{natbib}

\newcommand{\R}{\mathbb{R}}
\newcommand{\Gc}{\mathcal{G}}
\newcommand{\Hc}{\mathcal{H}}
\newcommand{\Tc}{\mathcal{T}}
\newcommand{\mult}{\mathrm{mult}}
\newcommand{\gmult}{\gamma\text{-}\mathrm{mult}}

\newcommand{\new}[1]{\textcolor{black}{#1}}

\newenvironment{NEW}
{\color{black}
}
{ 
  \color{black}
}

\usepackage{amsthm}
\newtheorem{theorem}{Theorem}[section]
\newtheorem{lemma}[theorem]{Lemma}
\newtheorem{definition}[theorem]{Definition}
\newtheorem{remark}[theorem]{Remark}
\newtheorem{corollary}[theorem]{Corollary}

\newtheorem{proposition}[theorem]{Proposition}

\begin{document}

\title{Nodal domain count for the generalized graph $p$-Laplacian}

\author{Piero Deidda}
\author{Mario Putti}
\address{Department of Mathematics ``Tullio Levi-Civita'',
  University of Padua,
  Italy}
\email{piero.deidda@math.unipd.it, mario.putti@unipd.it}

\author{Francesco Tudisco} %
\address{School of Mathematics, Gran Sasso Science Institute,
  Italy}
\email{francesco.tudisco@gssi.it}

\begin{abstract}
  Inspired by the linear Schr\"odinger operator, we consider a
  generalized $p$-Laplacian operator on discrete graphs and present
  new results that characterize several spectral properties of this
  operator with particular attention to the nodal domain count of its
  eigenfunctions.  Just like the one-dimensional continuous
  $p$-Laplacian, we prove that the variational spectrum of the
  discrete generalized $p$-Laplacian on forests is the entire
  spectrum. Moreover, we show how to transfer Weyl's inequalities for
  the Laplacian operator to the nonlinear case and prove new upper and
  lower bounds on the number of nodal domains of every eigenfunction
  of the generalized $p$-Laplacian on generic graphs, including
  variational eigenpairs. In particular, when applied to the linear
  case $p=2$, in addition to recovering well-known features, the new
  results provide novel properties of the linear Schr\"odinger
  operator.
\end{abstract}


\maketitle

\section{Introduction}
The study of nodal domains dates back to Sturm's oscillation theorem
that states that the zeros of the $k$-th mode of vibration of an
oscillating string are the endpoints of a partition of the string into
$k$ subdomains where the mode has constant sign.
Later Courant extended this result
to higher dimensions, proving that the $k$-th eigenfunction of an
oscillating membrane admits no more than $k$ subdomains, called nodal
domains \cite{Courant}. The count of the nodal domains of the
Laplacian operator, and its generalized version called the
Schr\"odinger operator, has been shown to hold important information
about the geometry of the system \cite{Blum, Gnutzmann1,
  Gnutzmann2}. In particular, nodal domains are tightly connected to
higher-order isoperimetric constants of graphs, making the study of
nodal domains particularly relevant in the context of data clustering,
expander graphs, and mixing time of Markov chains
\cite{alon1986eigenvalues, daneshgar2012nodal, fasino2014algebraic,
  lawler1988bounds}.  Because of these reasons, the estimation of the
number of nodal domains of the Schr\"odinger eigenfunctions both on
continuous manifolds and on discrete and metric graphs has been an
active field of research in recent years.

In the discrete graph setting, it was proved in \cite{Berkolaiko2,
  biy1} that trees behave like strings and that the $k$-th eigenvector
$f_k$ of the Schr\"odinger operator, if everywhere non-zero, induces
exactly $k$ nodal domains.  Moreover, again under the assumption that
the $k$-th eigenvector $f_k$ of the graph Schr\"odinger operator is
everywhere non-zero, it was proved in \cite{Berkolaiko1} that for
general graphs the following inequality holds for the number of nodal
domains $\nu(f_k)$ of $f_k$, provided the corresponding eigenvalue is
simple:
\begin{equation*}
  k-\beta+l(f_k)\leq \nu(f_n)\leq k\,.
\end{equation*}
Here $\beta$ is the total number of independent loops of the graph
and $l(f_k)$ is the number of independent loops where $f_k$ has
constant sign.
In the general case of eigenvalues with any multiplicity and
eigenvectors with possibly some zero entry, it was proved in
\cite{Davies, Duval, Xu} that
the following inequality holds:
  \begin{equation*}
  k+r-1-\beta-z\leq\nu(f_k)\leq k+r-1\,,
\end{equation*}
where $f_k$ is an eigenvector of the eigenvalue $\lambda_k$, $z$ is
the number of zeros of $f_k$ and $r$ is the multiplicity of
$\lambda_k$.

In recent years, there has been a surge in interest towards extensions
of the above results to the nonlinear $p$-Laplacian and
$p$-Schr\"odinger operators and their spectral properties, including
the nodal domain count, have been widely investigated both in the
continuous and in the discrete cases \cite{Amghibech2,Lindqvist1,
  Lindqvist2, drabek2002generalization}.  This interest is prompted by
applications to biological and physical processes often connected to
complex networks and data clustering,  \new{as well as  semi-supervised learning and machine learning in general \cite{Bhuler,Elmoataz, el2016asymptotic,flores2022analysis, calder2018game, slepcev2019analysis,prokopchik2022}, where the limiting
cases $p=1$ and $p=\infty$ are especially noteworthy}.  In particular,
similarly to the linear case, an important relation connects the nodal
domains of the $p$-Laplacian on graphs and the $k$-th order
isoperimetric constant $h_k$ of the graph.  Indeed, it is shown in
\cite{Tudisco1} that this fundamental graph invariant can be bounded
from above and from below using the variational spectrum $\lambda_k$
of the $p$-Laplacian and its nodal domain count via the Cheeger-like
inequality
\begin{equation}\label{eq:cheeger-p-lapl}
  \lambda_{\nu(f_k)}\leq h_{\nu(f_k)}\leq c(p) \lambda_k^{1/p}
\end{equation}
where $c(p)\to 1$ as $p\to 1$ and $f_k$ is any eigenfunction of
$\lambda_k$. To our knowledge, this result is the tightest available
connection between $h_k$ and the spectrum of the graph $p$-Laplacian
and clearly highlights the importance of the nodal domain count in
connection to, for example, the quality of $p$-Laplacian graph
embeddings for data clustering, for which there is a wealth of
empirical evidence
\cite{bresson2013multiclass,bresson2014multi,Bhuler,Elmoataz}.  A
nodal domain theorem for the graph $p$-Laplacian is provided in
\cite{Tudisco1}, where it is shown that, for any eigenfunction $f_k$
of the $p$-Laplacian, the number of nodal domains is bounded above as
$\nu(f_k)\leq k+r-1$, where $r$ is the multiplicity of the
corresponding eigenvalue. Analogous results are proved in
\cite{chang2017nodal} for the case $p=1$. However no lower bounds for
$\nu(f_k)$ are known in the general case.

The main aim of this paper is to provide lower bounds on the number of
nodal domains of the generic eigenfunction of the $p$-Laplacian. To
this end, we will introduce a class of generalized $p$-Laplacian
operators, already addressed in \cite{Park}. Such operators, inspired
by the generalized linear Laplacian or Schr\"odinger operator
\cite{Berkolaiko2,Xu}, are also largely related to $p$-Laplacian
problems with zero Dirichlet boundary conditions \cite{Hua}. Thus, all
our results apply to both the classical $p$-Laplacian and the
generalized $p$-Schr\"odinger operators. We prove a classical
characterization of the first and the second variational eigenpairs
and nonlinear Weyl's like inequalities. These are fundamental
instruments to study the nodal domains of the generic eigenfunction.
Our general strategy, inspired by the work of \cite{Berkolaiko2},
consists in defining appropriate rules to remove nodes or edges from
the graph without changing an eigenpair. Repeated applications of this
procedure allows us to arrive to a structured graph (e.g. a tree or
the disjoint union of the nodal domains) for which nodal domain
numbers or other spectral quantities can be fully characterized.  This
characterization can be brought back to the original graph by
reversing the proposed procedure.  This strategy allows us to find new
lower bounds as well as retrieve known upper bounds for the number of
nodal domains of any eigenfunction, as a function of the position of
the corresponding eigenvalue in the variational spectrum. In addition,
our estimates, with $p=2$, provide an improvement on the known results
for the linear case.

An important side result of our work is that we are able to prove
that, if the graph is a tree, the variational eigenvalues are all and
only the eigenvalues of the $p$-Laplacian operator and that the $k$-th
eigenfunction, if everywhere nonzero, admits exactly $k$ nodal
domains. This result extends what is already known in the particular
cases of the path graph \cite{Tudisco1} and the star graph
\cite{Amghibech2}, and is a generalization to the $p$-Laplacian of
analogous findings known in the linear case \cite{biy1,Berkolaiko2}.
This is of independent interest for its potential applications to
nonlinear spectral graph sparsification, expander graphs, and graph
clustering \cite{spielman2011spectral}. In particular, note that our
findings in combination with \eqref{eq:cheeger-p-lapl} show that the
$k$-th order isoperimetric constant of a tree coincides with the
$k$-th variational eigenvalue of the $1$-Laplacian.

The paper is organized as follows. In Section \ref{sec:notation} we
introduce notations and needed definitions. Our main results are
collected without proofs in Section \ref{sec:variational-eigs}. The
strategy for the proofs is organized into several steps. In Section
\ref{sec:preliminary-results} we provide some preliminary results
about the eigenfunctions of the generalized graph $p$-Laplacian,
including the characterization of the first two eigenvalues.  In
Section \ref{sec:weyls} we develop the procedures that remove nodes
and edges maintaining an eigenpair, and, along the way, we prove
Weyl's like inequalities.  In section \ref{sec:trees} we analyze the
particular case of a tree. Finally, in Section
\ref{sec:generic-graphs} we provide the proofs of the main results,
namely new inequalities on the number of the nodal domains of the
eigenfunctions of the graph $p$-Laplacian operator.

\section{Notation}\label{sec:notation}

Let $\mathcal{G}=(V,E)$ be a connected undirected graph, where $V$ and
$E$ are the sets of nodes and edges endowed with positive measures
$\varrho:V\to\mathbb R_+$ and $\omega:E\to\mathbb R_+$,
respectively. Given a function $f:V\to\mathbb{R}$, for any $p> 1$
consider the $p$-Laplacian operator:
\begin{equation*}
  (\Delta_p f)(u) := \sum_{v \sim u }
  \omega_{uv}|f(u)-f(v)|^{p-2}(f(u)-f(v)) \qquad \forall u\in V\, ,
\end{equation*}
where $v\sim u$ denotes the presence of an edge between $v$ and
$u$. We point out that, while the $p$-Laplacian is well-defined also
for $p=1$, throughout this work will not consider this limit case and
we will always implicitly assume that $p>1$.
Throughout the paper, we will often use the function $\phi_p(x):=|x|^{p-2}x$, so that the above equation reads simply:
\begin{equation*}
  (\Delta_p f)(u) := \sum_{v \sim u }
  \omega_{uv}\phi_p\left(f(u)-f(v)\right) \qquad \forall u\in V\, .
\end{equation*}

In analogy to the linear case, where the generalized Laplacian is
defined as the Laplacian plus a diagonal matrix \cite{Xu}, we define
the generalized $p$-Laplacian (or $p$-Schr\"odinger) operator as
\begin{equation*}
  (\Hc_p f)(u):= (\Delta_pf)(u)
  +\kappa_u |f(u)|^{p-2}f(u) \qquad \forall u\in V\,,
\end{equation*}
where $\kappa_u$ is a real coefficient.  We say that $f$ is an
eigenfunction of $\Hc_p$ if there exists $\lambda\in\mathbb{R}$ such
that
\begin{equation}\label{eq:p-L-eigprob}
  (\Hc_p f)(u)=\lambda\,\varrho_u |f(u)|^{p-2}f(u)
  \qquad \forall u\in V\,. 
\end{equation}
Similarly to the linear case, generalized $p$-Laplacians and their
eigenfunctions are directly connected with the solutions of Dirichlet
problems on graphs for the $p$-Laplacian operator. In fact, assume we
have a graph $\Gc=(V,E)$ with node and edge sets that can be
partitioned as the disjoint union of internal and boundary sets
$V=(V_I\sqcup V_B)$ and $E=(E_I\sqcup E_B)$, defined as
$E_I =\{uv\in E: u,u\in V_I\}$ and
$E_B = \{uv\in E: u\in V_I, v\in V_B\}$. This is the definition used
for example in \cite{Friedman1993}.  Then, if $f$ is a solution to the
Dirichlet problem
\begin{equation}\label{eq:p-L-dirichlet}
  \begin{cases}
    (\Delta_p f)(u)=\lambda\,\varrho_u |f(u)|^{p-2}f(u)
    &\quad \forall u\in V_I \\
    f(u)=0  & \quad \forall u\in{V}_B
  \end{cases}\, ,
\end{equation}
we deduce that $f$ is automatically also solution to the following
eigenvalue equation for a generalized $p$-Laplacian where the
information about the boundary nodes has been condensed in the nodal
weights:
\begin{multline*}
  \sum_{v\in{V}_{I}} \omega_{uv}|f(u)-f(v)|^{p-2}(f(u)-f(v))
  +\Big(\sum_{v\in{V}_{B}}\omega_{uv}\Big)|f(u)|^{p-2}f(u)\\
  =\lambda\,\varrho_u |f(u)|^{p-2}f(u)\,.
\end{multline*} 
In other words, the $p$-Laplacian Dirichlet problem with zero
boundary conditions is equivalent to the eigenvalue problem
\eqref{eq:p-L-eigprob} for the generalized $p$-Laplacian $\Hc_p$
with $\kappa_u = \sum_{v\in{V}_B}\omega_{uv}$.

Finally, the following definition introduces the concept of strong
nodal domains of $\Gc$, corresponding to a given function
$f:V\to\mathbb R$.
\begin{definition}[Nodal domains]\label{def:nodal-domains}
  Consider a graph $\Gc=(V, E)$ and a function
  $f:V\rightarrow \mathbb{R}$. A set of vertices
  $A\subseteq V$ is a nodal domain induced by $f$ if the subgraph
  $\Gc_A$ with vertices in $A$ is a maximal connected subgraph of
  $\Gc$ where $f$ is nonzero and has constant sign. For convenience,
  in the following we will refer interchangeably to both $A$ and
  $\Gc_A$ as the nodal domain induced by $f$.
  \end{definition}
Sometimes it is useful to distinguish between maximal subgraphs where
the sign is strictly defined and those where zero entries are
allowed. In particular, when zero entries of $f$ are allowed in the
definition above, the maximal subgraphs are called weak 
nodal domains, whereas the maximal subgraphs with strictly positive or
strictly negative sign as in Definition \ref{def:nodal-domains} are
called strong nodal domains.  However, as in this work we are not
interested in weak nodal domains, throughout we shall simply use the
term ``nodal domain'' to refer to the strong nodal domains, as defined
above.

\section{Variational spectrum and  main results}
\label{sec:variational-eigs} 

In this section we state our main results and will devote the
remainder of the paper to their proof. We first need to introduce the
notion of variational spectrum.
A set of $N$ variational eigenvalues of the generalized $p$-Laplacian
on the graph $\Gc = (V,E)$ can be defined via the
Lusternik--Schnirelman theory and the min-max procedure based on the
Krasnoselskii genus, which we review below \cite{Yiallourou}.

\begin{definition}[Krasnoselksii genus]\label{def:krasno}
  Let $X$ be a Banach space and consider the class $\mathcal A$ of
  closed symmetric subsets of $X$,
  $\mathcal A=\lbrace A \subseteq X| \; A\;\: closed\,,\;\:
  A=-A\rbrace\; .$
  For any $A\in\mathcal A$ consider the space of the Krasnoselskii
  test maps on $A$ of dimension $k$:
  \begin{equation*}
    \Lambda_k(A)=\lbrace \varphi:A\to \mathbb R^k
    \text {continuous and such that } \varphi(x)=-\varphi(-x)\rbrace. 
  \end{equation*}
  The Krasnoselskii genus of $A$ is the number $\gamma(A)$ defined as
  \begin{equation*}
    \gamma(A)=
    \begin{cases}
      \inf\lbrace k\in\mathbb{N}\,: \,
      \exists\,\varphi\in\Lambda_k(A)\;\, s.t. \;\,
      0\not\in\varphi(A) \rbrace &
       \\
 \new{    \infty \qquad \qquad \qquad \qquad \text{if $\nexists \; k$ as above}}
       \\
 \new{    0 \qquad \qquad \qquad \qquad\;\: \text{if $A=\emptyset$}}
    \end{cases}\, .
  \end{equation*}
\end{definition}
\new{Our reference Banach space is the space of vertex states $X=\{f:V\rightarrow \R\}=\R^N$ and  $\mathcal A$ denotes the family of all  closed symmetric subsets of $\R^N$. Let  $\mathcal S_p =\{f\in X:\|f\|_p=1\}$ be the $p$-unit sphere on $X$ and for  $1\leq k\leq N$ consider the family of closed symmetric subsets of $\mathcal S_p$ of genus greater than $k$
$$
\mathcal{F}_k(\mathcal S_p):=\lbrace A\subseteq \mathcal{A}\cap
\mathcal S_p \,|\, \gamma(A)\geq k \rbrace\,.
$$
}
In order to define the variational eigenvalues of $\Hc_p$, we consider the Rayleigh quotient functional
\begin{equation*}
  \mathcal{R}_{\Hc_p}(f)=
  \frac{\sum_{uv\in E} \omega_{uv}|f(u)-f(v)|^p
    +\sum_{u\in V}\kappa_u|f(u)|^p}{\sum_{u\in V}\varrho_u|f(u)|^p}\, .
\end{equation*}
As $\mathcal{R}_{\Hc_p}$ is positively scale invariant, i.e.\ $\mathcal{R}_{\Hc_p}(\alpha f)=\mathcal{R}_{\Hc_p}(f)$ for all $\alpha>0$, it is not difficult to observe that the eigenvalues and eigenfunctions of the generalized $p$-Laplacian operator are the critical values and the critical points of $\mathcal{R}_{\Hc_p}$ on $\mathcal S_p$.  The Lusternik-Schnirelman theory allows us to define a set of $N$ variational such critical values, via the following principle
\begin{equation}\label{Variational_eigenvalues}
  \lambda_k=\min_{A\in\mathcal{F}_k(\mathcal S_p)}\max_{f\in A}\,
  \mathcal{R}_{\Hc_p}(f)\, .
\end{equation}

\new{
We emphasize that the Krasnoselskii genus is a homeomorphism-invariant generalization to symmetric sets of the notion of dimension. In particular, if $A\in \mathcal A$ is the intersection of any subspace of dimension $k$ with $\mathcal{S}_p$, then $\gamma(A)=k$.  Moreover, note that  any $A$ such that $\gamma(A)\geq k$ contains at least $k$ mutually orthogonal functions (see e.g.\  \cite{Solimini}).
Therefore,  the definition  in \eqref{Variational_eigenvalues} is a generalization of the Courant-Fisher min-max characterization of the eigenvalues of a symmetric matrix, as $\mathcal{F}_k(\mathcal S_p)$ contains all subspaces of dimension greater than $k$. However, while  Courant-Fisher applies directly to the case $p=2$, linear subspaces alone are not sufficient to provide critical points in the general case $p\neq 2$. 
}

\subsection{Multiplicity and $\gamma$-multiplicity}

\begin{NEW}
Similarly to the case of symmetric matrices, we note that the variational eigenvalues $\{\lambda_k\}$ are by definition an increasing
sequence. This allows us to define a notion of multiplicity for variational eigenvalues:
\begin{definition}
  Let $\lambda_k$ be a variational eigenvalue of $\Hc_p$. If $\lambda_k$
appears $m$ times in the sequence of the variational eigenvalues
\begin{equation}\label{multiple_eig}
  \lambda_1\leq\lambda_2\leq\dots\leq
  \lambda_{k-1}<\lambda_{k}
  =\dots=
  \lambda_{k+m-1}<\lambda_{k+r}\leq\dots\leq\lambda_N\,.
\end{equation}
we say that $\lambda_k$ has multiplicity $m$ and we write $\mult_{\Hc_p}(\lambda_k)=m$ or simply $\mult(\lambda_k)=m$ when no ambiguity may occur.  
\end{definition}

The notion of multiplicity defined above applies only to  variational eigenvalues. In the case of a generic eigenvalue $\lambda$, we can use the Krasnoselskii genus to extend the notion of geometric multiplicity to the nonlinear setting:
\begin{definition}
Let $\lambda$ be an eigenvalue of $\Hc_p$. If 
\begin{equation}\label{Def_gmult}
\gamma\Big(\{f\in \mathcal S_p \,:\,\Hc_p(f)=\lambda |f|^{p-2}f\}\Big)=m
\end{equation}
we say that $\lambda$ has $\gamma$-multiplicity $m$ and we write $\gmult_{\Hc_p}(\lambda)=m$, or simply  $\gmult(\lambda)=m$ when no ambiguity may occur.
\end{definition}

Finally, we define simple eigenvalues

\begin{definition}
We say that $\lambda$ is a simple eigenvalue of $\Hc_p$ if $\lambda$ has a unique eigenfunction  $f\in \mathcal S_p$.
\end{definition}

Notice that the notions of multiplicity and $\gamma$-multiplicity do not coincide and an eigenvalue with $\gamma$-multiplicity equal to one is not necessarily simple. Viceversa, if  $\lambda$ is a simple  eigenvalue, then necessarily $\gmult(\lambda)=1$ and, if $\lambda$ is variational then also $\mult(\lambda)=1$. This result is a direct consequence of the next lemma, whose proof follows directly from  Lemma 5.6 and Proposition 5.3, Chapter II  of \cite{struwe}:
\begin{lemma}\label{lemma_genus_and_multiplicity}
If $\lambda$ is a variational eigenvalue, then 
$$\gmult(\lambda)\geq \mult(\lambda)\, .$$
\end{lemma}
Note that the inequality above implies, in particular,  that, to any variational eigenvalue $\lambda$, there correspond at least $\mult(\lambda)$ orthogonal eigenfunctions. 
Finally, we remark the following  direct consequence of Lemma\ref{lemma_genus_and_multiplicity}
\begin{corollary}\label{Corollary_gmul_greater_than_N}
Let $\Hc_p$ be the generalized $p$-Laplacian operator on a graph $\Gc$ with $N$ nodes. Let $\{\lambda_i\}_{i=1}^n$  be the variational eigenvalues of $\Hc_p$ counted without multiplicity, i.e. $\lambda_i\neq\lambda_j\;\forall i\neq j$. Then 
$$\sum_{i=1}^n \gmult(\lambda_i)\geq \sum_{i=1}^n \mult(\lambda_i)= N\,,$$
with the equality holding if and only if $\gmult(\lambda_i)=\mult(\lambda_i)$, for all $i=1,\dots,n$.
\end{corollary}  

\end{NEW}

\subsection{Main results}

We present below our main results.
Recalling the idea summarized in the introduction, our strategy for
counting nodal domains of generalized $p$-Laplacians is to come up
with algorithmic steps to remove vertices and edges from the original
graph in such a way that the original eigenpairs can be recovered from
the eigenpairs of the new graph. Since the proofs of our main results
require relatively long arguments, we state the results here and
devote the remainder of the paper to their proofs.  In particular,
after discussing in Sections \ref{sec:preliminary-results} and
\ref{sec:weyls} a number of preliminary observations and results,
which are of independent interest, Section \ref{sec:trees} will
provide proofs for Theorems \ref{thm:main1-variational-eig-tree} and
\ref{thm:main2-nodal-count-trees}, which deal with the special case of
trees and forests, whereas Section \ref{sec:generic-graphs} will
present the proofs of Theorems \ref{thm:main3-nodal-count-simple} and
\ref{thm:main4-nodal-count-general}, which address the case of general
graphs.

\begin{figure}\label{Fig1}
  \centering
  \begin{tikzpicture}[inner sep=1.5mm, scale=.5, thick]

    \node (1) at (0,2) [circle,draw] {1};
    \node (2) at (3,0) [circle,draw] {2};
    \node (4) at (3,4) [circle,draw] {4};
    \node (3) at (6,2) [circle,draw] {3};

    \draw [-] (1.south) -- (2.west);
    \draw [-] (1.north) -- (4.west);
    \draw [-] (4.east) -- (3.north);
    \draw [-] (4.south) -- (2.north);
    \draw [-] (2.east) -- (3.south);
    \node at (15,2) {\begin{minipage}{.65\textwidth}
        \begin{enumerate}[noitemsep] 
        \item $f=(1,1,1,1),\quad \lambda=0$
        \item $f=(1,0,-1,0),\quad \lambda=2$
        \item $f=(0,1,0,-1),\quad \lambda=2+2^{p-1}$
        \item $f=(1,0,1,-2^{\frac{1}{p-1}}),\quad \lambda=1+\big(1+2^{\frac{1}{p-1}}\big)^{p-1}$
        \item $f=(1,-1,1,-1), \quad \lambda=2^{p}$
        \end{enumerate}
      \end{minipage}};
  \end{tikzpicture}
  \caption{ Left: Example graph in which the corresponding generalized
    $p$-Laplacian $\Hc_p$ with $\omega_{uv}=\varrho_u=1$ and
    $\kappa_u=0$, for all $u,v = 1,\dots,4$, has at least one
    non-variational eigenvalue. Right: Set of five eigenfunctions and
    the corresponding eigenvalues, for $\Hc_p$.}
\end{figure}

Notice that, unlike linear operators, the variational spectrum does
not cover the entire spectrum of the generalized $p$-Laplacian \new{and, in general, establishing whether a certain eigenvalue is variational or not is still an open problem}.  For
example, Amghibech shows in \cite{Amghibech1} that the $p$-Laplacian
on a complete graph admits more than just the variational
eigenvalues. Another simple example for the setting $\omega\equiv 1$,
$\varrho\equiv 1$ and $\kappa\equiv 0$ is provided by Figure
\ref{Fig1}, while a more refined analysis of non-variational
eigenvalues is recently provided by Zhang in
\cite{zhang2021homological}.

Our first main result shows that the situation is different for the
special case of trees and, more in general, forests. In fact, as for
the standard linear case, we prove that when $\Gc$ is a forest, the
variational spectrum covers all the eigenvalues of the generalized
$p$-Laplacian. Here and in the following, we use the symbol
$\sqcup$ to denote disjoint union.

\begin{NEW}
\begin{theorem}\label{thm:main1-variational-eig-tree}
  Let $\Gc=\sqcup_{i=1}^k \Tc_i$ be a forest, $\Hc_p$ a generalized $p$-Laplacian
  operator on $\Gc$, $p>1$, and $\Hc_p(\Tc_i)$ the restriction of $\Hc_p$ to the $i$-th tree $\Tc_i$. 
  Then $\Hc_p$ admits only variational eigenvalues and for any such eigenvalue $\lambda$ it holds 
  $$\mult_{\Hc_p}(\lambda)=\gmult_{\Hc_p}(\lambda)=\sum_{i=1}^k \mult_{\Hc_p(\Tc_i)}(\lambda)$$
where $\mult_{\Hc_p(\Tc_i)}(\lambda)=0$ if $\lambda$ is not an eigenvalue of $\Hc_p(\Tc_i)$.
\end{theorem}
\end{NEW}
In addition, we are able to prove the following theorem about the
number of nodal domains induced on a forest, which generalizes
well-known results for the case of the linear Laplacian
\cite{Fiedler, Berkolaiko2, biy1}. 

\begin{theorem}\label{thm:main2-nodal-count-trees}
  Let $\Gc=\sqcup_{i=1}^m \Tc_i$ be a forest and consider the
  generalized $p$-Laplacian operator $\mathcal H_p$, $p>1$, on $\mathcal G$.
  If $f_k$ is an everywhere nonzero eigenfunction associated to
  the eigenvalue $\lambda_k=\dots=\lambda_{k+m-1}$ of $\Hc_p$, then
  $f_k$ changes sign on exactly $k-1$ edges. In other words, $f_k$
  induces exactly $k-1+m$ nodal domains.
\end{theorem}

Next, we address the case of general graphs.  A tight upper bound for
the number of nodal domains of the eigenfunctions of the $p$-Laplacian
on graphs is provided in \cite{Tudisco1}. It is not difficult to
  observe that the same upper bound carries over unchanged to the
  generalized $p$-Laplacian case. This is summarized in the following
  result.

\begin{theorem}\label{thm:main3-nodal-count-simple}
  Suppose that $\Gc$ is connected and
  $\lambda_1<\lambda_2\leq\dots\leq\lambda_N$ are the variational
  eigenvalues of $\Hc_p$, $p>1$. Let $\lambda$ be an eigenvalue of $\Hc_p$
  such that $\lambda < \lambda_k$. Any eigenfunction associated to
  $\lambda$ induces at most $k-1$ nodal domains.
\end{theorem}

Finally, the following theorem provides novel lower bounds for the
number of nodal domains of $\mathcal H_p$ in the case of general
graphs. Morever, when tailored to the case $p=2$, it provides improved
estimates of the nodal domain count that are strictly tighter than the
currently available results \cite{Xu,Berkolaiko1}.
We will discuss these properties in more details below.

\begin{theorem}\label{thm:main4-nodal-count-general}
  Suppose that $\Gc$ is a connected graph with $\beta=|E|-|V|+1$
  independent loops, and let $\lambda_1\leq \cdots \leq \lambda_N$
  be the variational eigenvalues of $\Hc_p$, $p>1$. 
  For a function $f:V\to\mathbb R$, let $\nu(f)$ be the
  number of nodal domains induced by $f$, $l(f)$ the number of
  independent loops in $\Gc$ where $f$ has constant sign and  \new{$\lbrace v_i \rbrace_{i=1}^{z(f)}$} the nodes such that
  $f(v_i)=0$, \new{with $z(f)$ being the number of such nodes}. Let
  $\Gc'=\Gc\setminus \lbrace v_i \rbrace_{i=1}^{z(f)}$ be the graph obtained by
  removing from $\Gc$ all the nodes where $f$ is zero
    as well as all the edges connected to those nodes. Let $c\new{(f)}$ be number of
  connected components of $\Gc'$ and $\beta'\new{(f)}=|E'|-|V'|+c\new{(f)}$ the number
  of independent loops of the graph $\Gc'$. Then:
  \begin{enumerate}
  \renewcommand{\theenumi}{P\arabic{enumi}}
  \renewcommand{\labelenumi}{\theenumi.}
  \item\label{it:main4-1} If $f$ is an eigenfunction of $\Hc_p$
    with eigenvalue $\lambda$ such that $\lambda>\lambda_{k}$, then
    $f$ induces strictly more than $k-\beta+l(f)-z(f)$ nodal
    domains. Precisely, it holds $\nu(f)\geq k-\beta'\new{(f)}+l(f)-z(f)+c(f)$.
  \item\label{it:main4-2} If $f$ is an eigenfunction of $\Hc_p$
    corresponding to the variational eigenvalue $\lambda_k>\lambda_{k-1}$ with 
    $\mult_{\Hc_p}(\lambda_k)=m$, 
    then $\nu(f)\geq k+\new{m}-1-\beta'\new{(f)}+l(f)-z(f)$.
  \end{enumerate}
\end{theorem}

Before moving on, we would like to briefly comment on the above
results and provide a comparison with respect to lower bounds
available for the linear case $p=2$.  First, note that
both~\ref{it:main4-1} and~\ref{it:main4-2} in
Theorem~\ref{thm:main4-nodal-count-general} apply to variational
eigenvalues of $\Hc_p$. However they are not corollaries of each other
in the sense that there are settings where~\ref{it:main4-1} is more
informative than~\ref{it:main4-2} and vice-versa. %
Indeed, if $\lambda_k$ is a variational eigenvalue of multiplicity equal to one, then $\lambda_k>\lambda_{k-1}$ and from~\ref{it:main4-1} we obtain
\begin{equation}\label{eq:low1}
  \nu(f) \geq k - \beta'\new{(f)} + l(f)-z(f) + (c\new{(f)}-1) 
\end{equation}
for any eigenfunction $f$ of $\lambda_k$, which is strictly tighter
than the lower bound in~\ref{it:main4-2}. However, in P2, when $\lambda_k$ has multiplicity $\new{m}>1$, we have $\lambda_k>\lambda_{\new{k-1}}$ and the two lower bounds in~\ref{it:main4-1} and ~\ref{it:main4-2} cannot be compared a-priori. Instead, their combination leads to
\begin{equation}\label{eq:low2}
  \nu(f) \geq
  \max\Big\{\Big(k-\beta'\new{(f)}+l(f)-z(f)+(c\new{(f)}-\new{1})\Big),
  \Big(k-\beta'\new{(f)}+l(f)-z(f)+(m-1)\Big) \Big\}
\end{equation}
for any eigenfunction $f$ of $\lambda_k$.
These observations allow us to draw new lower bounds for the
eigenvalues of $\Hc_2$, which are all variational. In fact, for a
simple eigenvalue $\lambda_k$ of $\Hc_2$ with an everywhere nonzero
eigenfunction $f$, it was proved in~\cite{Berkolaiko1} that
$\nu(f)\geq k-\beta+l(f)$. Point~\ref{it:main4-1} of
Theorem~\ref{thm:main4-nodal-count-general} improves this result by
allowing eigenfunctions with zero nodes via
inequality~\eqref{eq:low1}. Note that this implies in particular
$\nu(f)\geq k - \beta + l(f)-z(f)$, as $c>1$ and $\beta'(f)\leq \beta$. 
Similarly, when $\lambda_k$ is a multiple eigenvalue of multiplicity
$m$ and $f$ is any corresponding eigenfunction, it was proved
in~\cite{Xu} for the linear case that \new{$\nu(f)\geq k+m-1-\beta-z(f)$}.
Combining~\ref{it:main4-1} and~\ref{it:main4-2} allows us to improve
this bound via the sharper version given in \eqref{eq:low2}, which
further accounts for the number of independent loops of $f$, the
number of connected components of $\Gc'$ and its number of
independent loops.

\section{Preliminary properties of the eigenfunctions of the
  generalized $p$-Laplacian}\label{sec:preliminary-results}

In this section we present a brief review of the main results about
$p$-Laplacian eigenpairs and discuss how to extend them to the
generalized $p$-Laplacian case.
We start with the characterization of the first and the last
variational eigenvalues.  Classical results available for the
$p$-Laplacian equation in the continuous case \cite{Lindqvist1,
  Lindqvist2}, have been extended to the discrete case in
\cite{Tudisco1, Hua}.  In the following we present analogous results
for the generalized $p$-Laplacian operator on graphs.

\subsection{The smallest variational eigenvalue}
We consider in this section the first (smallest) variational
eigenvalue $\lambda_1$ of $\Hc_p$, defined as:
\begin{equation}\label{eq:lambda_1}
    {\lambda_1=\min_{f\in \mathcal{S}_p} \mathcal{R}_{\Hc_p}(f)}\, .
\end{equation}
Since obviously $\mathcal{R}_{\Hc_p}(f)\geq \mathcal{R}_{\Hc_p}(|f|)$
for all $f\in\mathcal{S}_p$, we can assume that the first
eigenfunction $f_1$ is always greater then or equal to zero. On the
other hand, if $f_1(u)=0$ for some $u\in V$, then from the eigenvalue
equation \eqref{eq:p-L-eigprob} we get
\begin{equation*}
  \Hc_p(f_1)(u)=-\sum_{v\in V}
  \Big(\omega_{uv}|f_1(v)|^{p-2}f_1(v)\Big)=0\,,
\end{equation*}
which shows that $f_1$ assumes both positive and negative values,
contradicting the previous assumption.
We deduce that any eigenfunction corresponding to $\lambda_1$ must be
everywhere strictly positive, i.e., $f_1(u)>0$ $\forall u$. This
observation generalizes a well-known result for the standard
$p$-Laplacian ($\kappa_u=0$) on a graph with no boundary for which
$\lambda_1=0$ and any corresponding eigenfunction is positive and has
constant values \cite{Amghibech1}.  We formalize the characterization
of the first eigenfunction of the generalized $p$-Laplacian in the
following theorem.
\begin{theorem}\label{Theorem0.1}
 Let $\lambda_1$ be the first eigenvalue of $\Hc_p$ on a
    connected graph $\Gc$ as in \eqref{eq:lambda_1}. Then
  \begin{enumerate}
  \item $\lambda_1$
    is simple and the corresponding eigenfunction $f_1$ is strictly
    positive, i.e., $f_1(u)>0 \; \forall u\in V$;
  \item if
    $g$ is an eigenfunction associated to an eigenvalue $\lambda$ of
    $\Hc_p$ and $g(u)>0 \; \forall u\in V$, then $\lambda=\lambda_1$.
  \end{enumerate}
\end{theorem}
\begin{proof}
  We have already observed that any eigenfunction $f$ of $\lambda_1$
  must be strictly positive so it remains to prove that for any
  strictly positive eigefunction $g$ associated to an eigenvalue
  $\lambda$, it holds $g=f_1$ and $\lambda=\lambda_1$. From the
  eigenvalue equation, we have
  \begin{align}
    &\sum_{v\sim u} \omega_{uv}\phi_p(f_1(u)-f_1(v))
      =\big(\lambda_1 \varrho_u-\kappa_u\big)\, f_1(u)^{p-1} \label{eq1.1}\\
    &\sum_{v\sim u} \omega_{uv}\phi_p(g(u)-g(v))
      =\big(\lambda \varrho_u-\kappa_u\big)\, g(u)^{p-1}\label{eq1.2}
  \end{align}
  where $\phi_p$ is defined in Section~\ref{sec:notation}.
  If we multiply both sides of \eqref{eq1.1} by the function
  $f_1(u)-g(u)^p f_1(u)^{1-p}$ and both sides of \eqref{eq1.2} by
  $g(u)-f_1(u)^p g(u)^{1-p}$, we obtain
  \begin{align*}
    &\sum_{v\sim u} \omega_{uv}\phi_p\big(f_1(u)\!-\!f_1(v)\big)
      \Big(\!f_1(u)\!-\!g(u)^p f_1(u)^{1-p}\!\Big)\!
      =\!\big(\lambda_1 \varrho_u-\!\kappa_u\big)
      \Big(\!f_1(u)^{p}-\!g(u)^p\!\Big),\\
    &\sum_{v\sim u}  \omega_{uv}\phi_p\big(g(u)-g(v)\big)
      \Big(g(u)-f_1(u)^p g(u)^{1-p}\Big)
      \!=\!\big(\lambda \varrho_u-\kappa_u\big) \Big(g(u)^p-f_1(u)^p\Big).
  \end{align*}
  Summing the two equations first together and then over all the
  vertices, we obtain
  \begin{equation}\label{eq001}
    S(f_1,g) + S(g,f_1) =
    \big(\lambda_1-\lambda\big)
    \sum_{u\in V}\varrho_u\Big(f_1(u)^p-g(u)^p\Big)
  \end{equation}
  with 
  \begin{equation*}
    S(f,g) = \sum_{uv\in
      E}\omega_{uv}\Bigg(|g(u)-g(v)|^p-\phi_p(f(u)-f(v))
      \Big(\frac{g(u)^p}{f(u)^{p-1}}-\frac{g(v)^p}{f(v)^{p-1}}\Big)\Bigg)\,.
  \end{equation*}
  If we apply Lemma \ref{Lemma0.1} to the above sums first with ${\alpha={f_1(u)}/{f_1(v)}>0}$ and then with
  $\alpha={g(u)}/{g(v)}>0$, we deduce that both $S(f_1,g)$ and
  $S(g,f_1)$ are non-negative.
  Thus, if $\lambda = \lambda_1$, in which case
  $S(f_1,g) = S(g,f_1)=0$, again using Lemma \ref{Lemma0.1}, we obtain
  \begin{equation*}
    \displaystyle{\frac{g(u)}{g(v)}=\frac{f_1(u)}{f_1(v)} } \, ,
  \end{equation*}
  which shows that, since the graph is connected, $g$ is proportional
  to $f_1$, implying $\lambda_1$ simple. This allows us to conclude that
  $f_1$ and $g$ are the same eigenfunction.  Assume now that there
  exists an eigenvalue $\lambda>\lambda_1$ with the associated
  eigenfunction $g$ being strictly positive. For any $\varepsilon >0$,
  the function $\varepsilon g$ is also a strictly positive
  eigenfunction associated with $\lambda$. Thus we can find a
  $\varepsilon>0$ such that $f_1(u)>\varepsilon{g(u)}$ for all
  $u\in{V}$. This yields an absurd in \eqref{eq001} as the left hand
  side term is strictly positive and the right hand side is strictly
  negative. Thus, every eigenfunction that does not change sign has to
  be necessarily associated to the first eigenvalue and this concludes
  the proof.
\end{proof}
The following corollary is a direct consequence of
Theorem~\ref{Theorem0.1} and generalizes to $\Hc_p$ a well-known
property of the eigenfunctions of the standard $p$-Laplacian (see
e.g.\ \cite[Cor.~3.6]{Tudisco1})
\begin{corollary}\label{cor:at-least-two-nd}
  The first eigenvalue $\lambda_1$ of the generalized $p$-Laplacian
  operator defined on a connected graph is simple, i.e.\
  $\lambda_1<\lambda_2$, and any eigenvector associated to an
  eigenvalue different from $\lambda_1$ has at least two nodal
  domains.
\end{corollary}

\subsection{The largest variational eigenvalue}

Opposite to the case of the first variational eigenvalue, the last
variational eigenvalue realizes the maximum of the Rayleigh quotient:
\begin{equation*}
  \lambda_N=\max_{f\in \mathcal{S}_p} \mathcal{R}_{\Hc_p}(f)
\end{equation*}
and, following \cite{Amghibech2}, one can provide an upper bound to
the magnitude of $\lambda_N$ in terms of $\omega,\varrho$ and the
potential $\kappa$.
\begin{proposition}
  The largest variational eigenvalue $\lambda_N$ of the generalized
  $p$-Laplacian operator $\Hc_p$ defined on a connected graph
  satisfies:
  \begin{equation*}
    |\lambda_N|\leq \max_{u\in V}\ \Big(2^{p-1}\sum_{v\sim u}
    \frac{\omega_{uv}}{\varrho_u} +
    \frac{|\kappa_{u}|}{\varrho_u}\Big)\, .
  \end{equation*}
\end{proposition}
\begin{proof}
  Let $f_N$ be an eigenfunction associated to $\lambda_N$ and let
  $u_0$ be a node where $\varrho f_N$ assumes the maximal absolute
  value $|\varrho_{u_0} f_N(u_0)|=\max_{v\in V}|\varrho_v f_N(v)|$.
  Then, from the eigenvalue equation, we have
  \begin{equation*}
    \varrho_{u_0}|\lambda_N||f_N(u_0)|^{p-1}
    =\Big|\sum_{v\sim u_0}\omega_{u_0v}\phi_p\big(f_N(u_0)-f_N(v)\big)
    +\kappa_{u_0}\phi_p\big(f_N(u_0)\big)\Big|
  \end{equation*}
  from which we  obtain 
  \begin{equation*}
  |\lambda_N|\leq
  \sum_{v\sim u_0} \frac{\omega_{u_0v}}{\varrho_{u_0}}2^{p-1}
  + \frac{|\kappa_{u_0|}}{\varrho_{u_0}}
  \leq \max_{u\in V}\ \big(2^{p-1}\sum_{v\sim u}
  \frac{\omega_{uv}}{\varrho_u}
  + \frac{|\kappa_{u}|}{\varrho_u}\big). \qedhere
  \end{equation*}
\end{proof}
As done for the first eigenfunction, we provide here a
characterization of the sign pattern of the last (maximal)
eigenfunction in the particular case of bipartite graphs. Our result
extends to the generalized $p$-Laplacian the analogous results
obtained in the linear case in \cite{Oren,biy2} and in the case of the
$p$-Laplacian with Dirichlet boundary conditions in \cite{Hua}.
\begin{NEW}
\begin{theorem}\label{Theorem0.3}
  If $\Gc$ is a bipartite connected graph, then the largest eigenvalue $\lambda_N$ of $\Hc_p$ is simple and the corresponding unique eigenfunction $f_N$ is such that
  $f_N(u)f_N(v)<0$, for any $u\sim v$.
\end{theorem}
\end{NEW}
\begin{proof}
  We start by proving that if $f\in\mathcal{S}_p$ is a maximizer of
  the Rayleigh quotient, necessarily $f(u)f(v)<0$, $\forall u\sim v$.
  Indeed, since $\Gc$ is a bipartite graph we can decompose $V$
    into two subsets $V=V_1\sqcup V_2$, such that if
  $u,v\in V_i$, $i=1,2$, then $u\not\sim v$. Thus, starting from $f$,
  we define $f'$ such that $f'(u)=|f(u)|$, $\forall u\in V_1$ and
  $f'(u)=-|f(u)|$, $\forall u\in V_2$. Now observe that
  \begin{align*}
    \begin{aligned}
      \mathcal{R}_{\Hc_p}(f)&=
      \sum_{uv\in E} \omega_{uv}|f(u)-f(v)|^p+\sum_{u\in V} \kappa_u|f(u)|^p \\
      &\leq \sum_{uv\in E} \omega_{uv}\big||f(u)|+|f(v)|\big|^p
         +\sum_{u\in V} \kappa_u|f(u)|^p=\mathcal{R}_{\Hc_p}(f')
    \end{aligned}
  \end{align*}
  where the equality holds if and only if $f=\pm f'$. Since $f$ is a
  maximal eigenfunction, then $f=f'$ up to a sign and thus
  $f(u)f(v)\leq 0$, $\forall u\sim v$.  To conclude, if $f'(u)=0$
  then, for $u\in V_1$ we have $\lambda_n f'(u)=\Hc_p(f')(u)\leq0$ and
  the equality holds only if $f'(v)=0$ for every $v\sim u$. Since the
  graph is connected this would lead to the absurd $f'\equiv0$. Thus,
  we have that $f'(u)\neq 0$, $\forall u$, implying $f(u)f(v)< 0$,
  $\forall u\sim v$.

  We now prove uniqueness of the maximizer. Given two maximizers
  $f,g\in \mathcal{S}_p$ such that
  $$\mathcal{R}_{\Hc_p}(f)=\lambda_n=\mathcal{R}_{\Hc_p}(g)\,,$$
  up to a sign as above, $f$ and $g$ must be strictly
  greater than zero on $V_1$ and strictly smaller than zero on $V_2$.
  Then, similarly to the proof of Theorem~\ref{Theorem0.1}, we first
  multiply the eigenvalue equations for $f$ and $g$ by
  ${f(u)-{|g(u)|^p}/{\phi_p(f(u))}}$ and
  ${g(u)-{|f(u)|^p}/{\phi_p(g(u))}}$, respectively. Then, we sum the
  two equations together and over all the nodes to obtain:
  \begin{multline*}
    \sum_{uv\in E}\omega_{uv}\bigg(|g(u)-g(v)|^p
    -\phi_p\big((f(u)-f(v)\big)
    \Big(\frac{|g(u)|^p}{\phi_p\big(f(u)\big)}
    -\frac{|g(v)|^p}{\phi_p\big(f(v)\big)}\Big)\bigg)+\\
    \sum_{uv\in
      E}\omega_{uv}\bigg(|f(u)-f(v)|^p-\phi_p\big(g(u)-g(v)\big)
    \Big(\frac{|f(u)|^p}{\phi_p\big(g(u)\big)}
    -\frac{|f(v)|^p}{\phi_p\big(g(v)\big)}\Big)\bigg)=0
  \end{multline*}
  From Lemma \ref{Lemma0.1}, both the sums above are smaller than zero
  unless $f=g$, thus showing uniqueness of the maximizer and hence of
  the maximal eigenfunction $f_N$.
\end{proof}

\begin{corollary}\label{Theorem0.4}
  Consider a graph $\Gc$ and the generalized $p$-Laplacian operator
  $\Hc_p$. Then, the graph $\Gc$ is bipartite and connected if and
  only if the maximal eigenfunction $f_N$ of $\Hc_p$ induces exactly
  $N$ nodal domains.
\end{corollary}
\begin{proof}
  If the graph is bipartite, by Theorem \ref{Theorem0.3} the $N$-th
  variational eigenfunction is unique and induces $N$ nodal
  domains. Vice-versa, let $f_N$ be an eigenfunction such that $f_N$
  induces exactly $N$ nodal domains. Then, considering
  $V_1=\lbrace v| f_N(v)>0 \rbrace$ and
  $V_2=\lbrace v| f_N(v)<0 \rbrace$, we have $V=V_1\sqcup V_2$
  and each node in $V_1$ is connected only to nodes in $V_2$, showing
  that the graph is bipartite.
  \end{proof}

\subsection{Further properties of $\Hc_p$ and its  eigenfunctions}
Observe that, similarly to the linear Schr\"odinger operator and
unlike the $p$-Laplacian case, the eigenvalues of the generalized
$p$-Laplacian depend on the potential $\kappa_u$ and may attain both
positive and negative values. This follows directly from the
eigenvalue equation \eqref{eq:p-L-eigprob} for $(\lambda_1,f_1)$:
\begin{equation*}
  \sum_{v\sim u}
  \Big(\omega_{uv}|f_1(u)-f_1(v)|^{p-2}(f_1(u)-f_1(v))\Big)+\kappa_u
  f_1(u)^{p-1}=\lambda_1\,\varrho_u f_1(u)^{p-1}\,.
\end{equation*}
In fact, summing over all the vertices $u\in{V}$ yields
\begin{equation*}
  \lambda_1 =\frac{\sum_{u\in V}\kappa_u f_1(u)^{p-1}}%
  {\sum_{u\in V}\varrho_u f_1(u)^{p-1}}
  =\frac{\sum_{u\in V}\frac{\kappa_u}{\varrho_u}\varrho_u f_1(u)^{p-1}}%
  {\sum_{u\in V}\varrho_u f_1(u)^{p-1}}
\end{equation*}
which shows that $\lambda_1$ is in the convex hull of the
coefficients $\{\frac{\kappa_u}{\varrho_u}\}$ and, since
$\kappa_u$ may be negative, $\Hc_p$ may not be positive definite.

The next lemmas extend to the generalized $p$-Laplacian the results
proved in \cite{Tudisco1} for the standard $p$-Laplacian, and provide
partial orderings for the given eigenpairs. In particular, Lemma
  \ref{Lemma0} below follows directly by replacing the standard
  $p$-Laplacian with the generalized operator $\Hc_p$ in the proof of
  \cite[Lemma 3.8]{Tudisco1} and, for this reason, its proof is
  omitted.

\begin{lemma}\label{Lemma0}
  If $f$ is an eigenfunction relative to an eigenvalue $\lambda$ and
  $A_1,\dots,A_m$ are the nodal domains of $f$, consider $f|_{A_i}$
  the function that is equal to $f$ on $A_i$ and zero on
  $V\setminus A_i$. Then
  \begin{equation*}
    \max \Big\{\mathcal R_{\Hc_p}(f) : f \in
    \mathrm{span}\{f|_{A_1},\dots,f|_{A_m}\} \Big\} \leq \lambda \, .
  \end{equation*}
\end{lemma}

\begin{NEW}
\begin{corollary}\label{Corollary0.2}
  If $f$ is an eigenfunction relative to an eigenvalue $\lambda$ and $f$ induces $k$ nodal domains, then $\lambda \geq \lambda_k$.
\end{corollary}
\end{NEW}
\begin{proof}
  If $A_1,\dots,A_k$ are the nodal domains of $f$, consider $f|_{A_i}$
  the function that is equal to $f$ on $A_i$ and zero on
  $V\setminus A_i$. If
  $\mathcal{A}=\mathrm{span}\lbrace f|_{A_1}, \dots, f|_{A_k}
  \rbrace$, then notice that the Krasnoselskii genus of
    $\mathcal A$ is $k$, i.e., $\gamma(\mathcal{A})=k$. Thus, from Lemma
  \ref{Lemma0}, we have that
  $\lambda_k=\min_{A\in\mathcal{F}_k}\max_{f\in A}\mathcal{R}_{\Hc_p}
  \leq \max_{f\in \mathcal{A}}\mathcal{R}_{\Hc_p}(f)\leq \lambda$.
\end{proof}

We conclude by noticing that, combining
Corollaries~\ref{cor:at-least-two-nd} and~\ref{Corollary0.2}, one
immediately obtains that, as for the standard $p$-Laplacian, the
second variational eigenvalue $\lambda_2$ of the generalized
$p$-Laplacian is the smallest eigenvalue larger than
$\lambda_1$. Precisely, it holds:
\begin{equation*}
  \lambda_2=\min\{\lambda : \lambda
  > \lambda_1 \text{ is an eigenvalue of }\Hc_p\}
\end{equation*}

\section{Graph perturbations and Weyl's-like inequalities}
\label{sec:weyls}
In this section we show how to modify the graph and, consequently, the
associated generalized $p$-Laplacian operator, maintaining eigenpairs.
In particular, we will show how to remove edges and nodes obtaining a
new generalized $p$-Laplacian operator on a simpler graph written as
a ``small'' perturbation of the initial operator
$\Hc_p$. For this perturbed operator, we will prove Weyl's like
inequalities relating its variational eigenvalues to those of the
starting operator.

\subsection{Removing an edge}\label{sec:remove-edge}
Consider a graph $\Gc$ and the generalized $p$-Laplacian operator
$\Hc_p$ on $\Gc$. Let $\lambda$ and $f$ be an eigenvalue and a
corresponding eigenfunction of $\Hc_p$ and let $e_0=(u_0,v_0)$ be an
edge of the graph such that $f(u_0)f(v_0)\neq 0$.  We want to define a
new generalized $p$-Laplacian operator $\Hc'_p$ on the graph
$\Gc':=\Gc\setminus e_0 $, such that $(f,\lambda)$ is also an
eigenpair of $\Hc_p'$.

Our strategy extends to the nonlinear case the work of
~\cite{Berkolaiko2}, where the new operator $\Hc_p'$ is written as a
rank-one variation of the starting Laplacian.
To this end, we write $\Hc'_p=\Hc_p+\Xi_p$ where
\begin{equation}\label{eqXi_p}
  (\Xi_p g)(u)=\begin{cases}
    0 & \text{if $u\neq\;u_0,v_0$}\\
    \omega_{u_0v_0}\Big(\phi_p(1-\alpha)\phi_p\big(g(u_0)\big)
    -\phi_p\big(g(u_0)-g(v_0)\big)\Big) & \text{if $u=u_0$}\\
    \omega_{u_0v_0}\Big(\phi_p(1-\frac{1}{\alpha})\phi_p\big(g(v_0)\big)
    -\phi_p\big(g(v_0)-g(u_0)\big)\Big) & \text{if $u=v_0$}
  \end{cases}\, ,
\end{equation}
$\alpha:={f(v_0)}/{f(u_0)}$ and $\phi_p(x):=|x|^{p-2}x$ as before.
It can be easily proved that $\Hc'_p$ is obtained from $\Hc_p$ by
considering the edge weights $\omega'$ given by
$\omega'_{uv}=\omega_{uv}$ if $(uv)\neq (u_0v_0)$ and
$\omega'_{u_0v_0}=0$. Thus $\Hc'_p$ can be seen as a generalized
$p$-Laplacian operator on a graph $\mathcal{G'}$ that is obtained from
$\mathcal{G}$ by deleting the edge $e_0$ and that acts on the nodes
that are not adjacent to $e_0$ exactly as $\Hc_p$ does.  Observe that
$\Hc_p'$ depends on the original eigenfunction $f$ and a direct
computation shows that $(\lambda,f)$ is still an eigenpair of the new
operator.

Now we want to compare the variational eigenvalues of $\Hc'_p$ with
the ones of $\Hc_p$ with ordering purposes. We first write the
Rayleigh quotient of the new operator $\Hc'_p$ as
\begin{equation*}
  \mathcal{R}_{\Hc'_p}(g)=\mathcal{R}_{\Hc_p}(g)+\mathcal{R}_{\Xi_p}(g)\,,
\end{equation*}
where, for $g\in\mathcal{S}_p$:
\begin{multline}\label{eqxi2}
  \frac{\mathcal{R}_{\Xi_p}(g)}{\omega_{u_0v_0}}
  =\left(\frac{|g(u_0)|^{p}}{\phi_p(f(u_0))}
    -\frac{|g(v_0)|^{p}}{\phi_p(f(v_0))}\right)\phi_p\big(f(u_0)  -f(v_0)\big)\\
  -\big(g(u_0)-g(v_0)\big)\phi_p\Big(g(u_0)-g(v_0)\Big)\, .
\end{multline}
A direct application of Lemma \ref{Lemma0.1}, shows that
$R_{\Xi_p}$ is positive if $\frac{f(v_0)}{f(u_0)}$ is negative and
negative if $\frac{f(v_0)}{f(u_0)}$ is positive.  Moreover, if we
assume that $g(v_0)$ and $g(u_0)$ are non zero, we can write $\Xi_p g$
in the following equivalent way
\begin{equation*}
  (\Xi_p g)(u)=\begin{cases}
    0 & \text{if $u\neq\;u_0,v_0$}\\
    \omega_{u_0v_0}\phi_p\big(g(u_0)\big)
    \Big(\phi_p(1-\alpha)-\phi_p(1-\frac{g(v_0)}{g(u_0)}\Big)
    & \text{if $u=u_0$}\\
    \omega_{u_0v_0}\phi_p\big(g(v_0)\big)
    \Big(\phi_p(1-\frac{1}{\alpha})-\phi_p(1-\frac{g(u_0)}{g(v_0)})\Big)
    & \text{if $u=v_0$}
  \end{cases}\,.
\end{equation*}
From this last equation we can easily see that, if
$g(v_0)=\alpha g(u_0)$, then $\Xi_p g=\mathcal{R}_{\Xi_p}(g)=0$.

To continue, we need the following lemma from \cite{Solimini},
reported here without proof, which provides a bound on the
Krasnoselskii genus of the intersection of different subsets.

\begin{lemma}{\cite[Prop.\ 4.4]{Solimini}}\label{Lemma1}
  Let $X$ be a Banach space and $\mathcal A$ the class
    of the closed symmetric subsets of $X$. Given $A\in \mathcal{A}$,
    consider a Karsnoselskii test map $\varphi\in\Lambda_k(A)$ with
    $k<\gamma(A)$. Then,
  $\gamma(\varphi^{-1}(0))\geq \gamma(A)-k$.
\end{lemma}

Using the fact that $R_{\Xi_p}$ is zero on the hyperplane
$\pi=\lbrace g: g(u_0)f(v_0)-g(v_0)f(u_0)=0\rbrace$, we obtain the
following ordering of the $k$-th eigenvalue of $\Hc_p$ within the
spectrum of $\Hc_p'$.

\begin{lemma}\label{Lemma1.2}
  Assume that there exist an eigenfunction $f$ of $\Hc_p$ and an edge
  $e_0=(u_0,v_0)$ such that $f(u_0), f(v_0)\neq 0$ and consider the
  operator $\Hc'_p=\Hc_p+\Xi_p$, where $\Xi_p$ is defined as in
  \eqref{eqXi_p}.  Let $\eta_k$ be the variational eigenvalues of
  $\Hc'_p$ and $\lambda_k$ those of $\Hc_p$. The following
  inequalities hold:
  \begin{itemize}
  \item If $\frac{f(v_0)}{f(u_0)}<0$,
    then $ \eta_{k-1} \leq \lambda_k \leq \eta_{k}$;
  \item If $\frac{f(v_0)}{f(u_0)}>0$,
    then $ \eta_{k} \leq \lambda_k \leq \eta_{k+1}$.
  \end{itemize}
  
\end{lemma}

\begin{proof} 
  Let $\mathcal F_k$ be the Krasnoselskii family
  $\mathcal F_k = \{A\subseteq \mathcal A\cap \mathcal S_p |
  \gamma(A)\geq k\}$ as defined in Section
  \ref{sec:variational-eigs}. Let $A_k\in\mathcal{F}_k$ be such that
  $\lambda_k=\max_{f\in A_k}\mathcal R_{\Hc_p}(f)\,,$ and let
  \begin{equation*}
    \pi=\lbrace g: g(u_0)f(v_0)-g(v_0)f(u_0)=0\rbrace\,.
  \end{equation*}
  Then $A_k\cap\pi=\phi|_{A_k}^{-1}(0)$, and from Lemma \ref{Lemma1},
  since $\phi|_{A_k}\in\Lambda_1(A_k)$, we have
  \begin{equation*}
    \gamma(A_k\cap\pi) \geq \gamma(A_k)-1\geq k-1\,.
  \end{equation*}
  Thus, $A_k\cap\pi\in\mathcal{F}_{k-1}$ and
  \begin{equation*}
    \eta_{k-1}=\min_{A\in\mathcal{F}_{k-1}}\max_{f\in A}
    \mathcal{R}_{\Hc'_p}(f)\leq \max_{f\in A_k\cap\pi} \mathcal
    R_{\Hc'_p}(f)=\max_{f\in A_k\cap\pi} \mathcal R_{\Hc_p}+\mathcal
    R_{\Xi_{p}}\leq\lambda_k\, .
  \end{equation*}
  This implies that  $\eta_{k-1}\leq\lambda_k$. 
  Moreover, since $\mathcal R_{\Xi_{p}}\geq0$ we have
  that $\mathcal R_{\Hc'_p}(f)\geq \mathcal R_{\Hc_p}$,  
  which implies
  \begin{equation*}
    \lambda_k=\min_{A\in\mathcal{F}_{k}}\max_{f\in A}
    \mathcal{R}_{\Hc_p}(f)\leq \min_{A\in\mathcal{F}_{k}}\max_{f\in A}
    \mathcal{R}_{\Hc'_p}(f)=\eta_k\,,
  \end{equation*}
  and this concludes the proof of the first inequality.
  The second inequality can be proved
  analogously, by exchanging the roles of $\Hc'_p$ and $\Hc_p$.
\end{proof}

\subsection{Removing a node}\label{sec:remove-a-node}
Consider a generalized $p$-Laplacian operator $\Hc_p$ defined on a
graph $\Gc$ and let $u_0$ be a node of $\Gc$. We want to define a new
operator $\Hc'_p$ on the graph $\Gc':=\Gc\setminus\lbrace u_0\rbrace$
that behaves on $\Gc'$ like $\Hc_p$ behaves on the hyperplane
$\lbrace f\,:\,f(u_0)=0\rbrace$. Note that, in the linear case,
this operation is equivalent to considering the principal submatrix of
the generalized Laplacian matrix obtained by removing the row and the
column relative to $u_0$.  If we remove a node $u_0$, we have to
remove also all its incident edges from the graph $\Gc$.  Thus, on the
graph $\Gc'$ we can define the generalized $p$-Laplacian:
\begin{equation}\label{eq:Hp-without-node}
  \Hc_p'(f)(u):= \sum_{v\in V'}\omega'_{uv}\phi_p(f(u)-f(v))
  +\kappa'_u\phi_p(f(u))\,,
\end{equation}
where $V'=V\setminus\{u_0\}$, $\omega'_{uv}=\omega_{uv}$ and
$\kappa'_u=\kappa_u+\omega_{uu_0}$.
\begin{NEW}
\begin{remark}\label{Remark_removing_a_node}
If $f$ is an eigenfunction of the generalized $p$-Laplacian $\Hc_p$
on $\Gc$, with eigenvalue $\lambda$ and such that $f(u_0)=0$, then the
restriction $f'$ of $f$ on the graph $\Gc'=\Gc\setminus\{u_0\}$ is automatically an
eigenfunction of $\Hc_p'$ with eigenvalue $\lambda$. Indeed, for each
$u\neq u_0$ we have
\begin{equation*} \lambda \varrho_u \phi_p(f(u)) =\sum_{v\neq
    u_0}\omega_{uv}\phi_p(f(u)-f(v))
  +\omega_{uu_0}\phi_p(f(u))+\kappa_u\phi_p(f(u))=\Hc_p'(f)(u)\, .
\end{equation*}
\end{remark}
\end{NEW}

Next, we provide an ordering for the variational eigenvalues of
$\Hc'_p$, in comparison with those of $\Hc_p$, as stated in the
following lemma.

\begin{lemma}\label{Lemma2}
  Given a node $u_0$ of $\Gc$, let $\Hc_p$ and $\Hc'_p$ be generalized
  $p$-Laplacian operators defined on the graphs $\Gc$ and
  $\Gc' =\Gc\setminus\{u_0\}$, respectively, and let $\lambda_k$ and
  $\eta_k$ be the corresponding variational eigenvalues. Then:
  \begin{equation*}
    \lambda_k\leq\eta_k\leq\lambda_{k+1}\,.
  \end{equation*}
\end{lemma}
\begin{proof}
  Let $\mathcal S_p'=\{f:V'\to\mathbb R:\|f\|_p=1\}$ and consider
  $A_k'\in\mathcal{F}_k(\mathcal S_p')$ such that
  \begin{equation*}
    \eta_k=\max_{f\in A_k'}\mathcal{R}_{\Hc'_p}(f)
    =\max_{f\in A_k'}\sum_{(uv)\in E'} w'(uv)|f(u)-f(v)|^p
    +\sum_{u\in V'}\kappa'_u|f(u)|^p\,,
  \end{equation*}
  where $E'$ is the set of edges of $\Gc'$.  Consider now $A_k$, the
  immersion of $A_k'$ in the $N-1$ dimensional hyperplane
    $\pi=\{f:V\to \mathbb R:f(u_0)=0\}$, i.e.\ the set of functions
  $f$ that, when restricted to the nodes different from $u_0$, belong
  to $A_k'$ and are such that $f(u_0)=0$. Thus, $A_k$ belongs to
  $\mathcal{F}_k(\mathcal S_p)$ since $A_k$ and $A_k'$ are
  homeomorphic, and we obtain:
  \begin{equation*}
    \lambda_k=\min_{A\in\mathcal{F}_k} \max_ {f\in  A}\mathcal{R}_{\Hc_p}(f)
    \leq \max_{f\in A_k}\mathcal{R}_{\Hc_p}(f)
    =\max_ {f\in A_k'}\mathcal{R}_{\Hc_p'}(f)=\eta_k\,.
  \end{equation*}
  To prove the other inequality, consider
  $A_{k+1}\in\mathcal{F}_{k+1}(\mathcal S_p)$ such that
  \begin{equation*}
    \lambda_{k+1}=\max_{f\in A_{k+1}}\mathcal{R}_{\Hc_p}(f)\,.
  \end{equation*}
  Because of Lemma \ref{Lemma1}, we have that
  $\gamma(A_{k+1}\cap\lbrace f: f(u_0)=0 \rbrace ) \geq k$, which
  implies that
  $A_{k+1}\cap\lbrace f: f(u_0)=0\rbrace\in\mathcal{F}_k(\mathcal
  S_p)$.  Thus
  \begin{equation*}
    \eta_k\leq \max_{f\in A_k'} \mathcal{R}_{\Hc_p'}(f)=\max_{f\in A_{k+1}\cap
      \pi } \mathcal{R}_{\Hc_p}(f)  \leq
    \max_{f\in A_{k+1}} \mathcal{R}_{\Hc_p}(f)=\lambda_{k+1}\, ,
  \end{equation*}
  where $A_k'$ is the set of functions $f':V'\to\mathbb R$ obtained as
  the restriction of functions from $A_{k+1}\cap\pi$ to $\Gc'$, i.e.,
  $f'\in A_k'$ if the lifting $f:V\to\mathbb R$ defined as $f(u_0)=0$
  and $f(u)=f'(u)$, $\forall u\neq u_0$, belongs to $A_{k+1}\cap\pi$.
\end{proof}

This result can be generalized by induction to the case of $n$ removed
nodes, obtaining the main theorem of this section.
\begin{theorem}\label{Theorem2.1}
    Let $\Hc_p$ be the generalized $p$-Laplacian operator defined on
    the graph $\Gc$ and let $\Gc'$ be the graph obtained from $\Gc$ by
    deleting the $n$ nodes $u_1,u_2,\dots,u_n$. Consider the
    generalized $p$-Laplacian operator on $\Gc'$ defined as
  \begin{equation*}
    \Hc_p'(u):=\sum_{v\in V'}\omega'_{uv}\phi_p(f(u)-f(v))
    +\kappa'_u\phi_p(f(u))\,,
  \end{equation*}
  where $V'=V\setminus \lbrace u_1,\dots, u_n\rbrace$,
  $\omega'_{uv}=\omega_{uv}$ and
  $\kappa'_u=\kappa_u+\sum_{i=1}^n\omega_{uu_i}$. Let
  $\lbrace \lambda_k \rbrace$ denote the variational eigenvalues of
  $\Hc_p$ and $\lbrace \eta_k \rbrace$ those of $\Hc_p'$. Then
  \begin{equation*}
    \lambda_k\leq\eta_k\leq\lambda_{k+n}\, ,
  \end{equation*}
  for any $k\in\lbrace1,\dots,|V|-n\rbrace$.
\end{theorem}
\begin{proof}
  The proof follows directly from Lemma \ref{Lemma2}, removing
  recursively the nodes $u_1,\dots,u_n$.
\end{proof}

\def\tilde{\widetilde}

\section{Nodal domain count on trees
  (proofs of Theorems \ref{thm:main1-variational-eig-tree}
  and \ref{thm:main2-nodal-count-trees})}\label{sec:trees}

In this section we deal with the case in which $\Tc:=\Gc=(V, E)$ is a tree
and we provide proofs of the two Theorems
\ref{thm:main1-variational-eig-tree} and
\ref{thm:main2-nodal-count-trees}. In particular, we will prove that
the eigenvalues of the generalized $p$-Laplacian on a tree are all and
only the variational ones. Moreover, again restricting ourselves to
trees, we will show that, if an eigenfunction of the $k$-th
variational eigenvalue is everywhere non zero, then it induces exactly
$k$ nodal domains. This generalizes to the nonlinear case a well-known
result for the linear Shr\"odinger operator.

\new{In the following, given a tree $\Tc = (V, E)$, we assume a root $r\in V$ is chosen arbitrarily. This provides a partial ordering of the nodes so that a precise root is automatically assigned to any subtree of $\Tc$. In particular, we write $v<u$ if $v$ is a descendant of $u$ and $v\prec u$ if $v$ is a direct child of $u$. Moreover, for each node $u\in V$, we let   $\Tc_u$ denote the subtree of $\Tc$ having $u$ as root and formed by  all the descendants of $u$.  On this subtree we can define a new operator $\Hc_p^{u}$ obtained as follows: starting from $\Tc$, we remove all the nodes that do not belong to $\Tc_u$ and, for each deleted node, we modify the original operator $\Hc_p$ on $\Tc$ as in Section \ref{sec:remove-a-node}.}

We also consider the operator $\Hc_p^{\tilde{u}}$, obtained by removing from $\Tc_u$ also the root node $u$ and by modifying $\Hc_p^u$ accordingly.
\new{This latter operator is defined on a subforest, $\Tc_{\tilde{u}}=\sqcup_i \Tc_i$,} that has as many connected components as the number of children of $u$.  From the generalized Weyl's inequalities of Section \ref{sec:weyls}, we have that 
\begin{equation}\label{eq1}\cdots\leq 
  \lambda_i(\Hc_p^u)
  \leq\lambda_i(\Hc_p^{\tilde{u}})
  \leq\lambda_{i+1}(\Hc_p^u)\leq \cdots
  \,
\end{equation}
where $\lambda_i(\Hc_p^u)$ and $\lambda_i(\Hc_p^{\tilde{u}})$ denote the $i$-th variational eigenvalue of $\Hc_p^u$ and $\Hc_p^{\tilde{u}}$, respectively. 
Observe also that
\new{$\displaystyle\Hc_p^{\tilde{u}}=\mathop{\oplus}_{v_i\prec u}{\Hc}_p(\Tc_i)$, where ${\Hc}_p(\Tc_i)$ is the generalized $p$-Laplacian of $\Tc_i$} and 
$v_i\prec u$ indicates that $v_i$ is a direct child of
$u$.

\subsection{\new{Generating functions}}
Consider now an eigenfunction $f$ of $\Hc_p$ with eigenvalue
$\lambda$ and assume that $f\neq 0$ everywhere. For each $u$
different from the root $r$, we denote by $u_F$ the parent of $u$ in $\Tc$. Then, the following quantity
\begin{equation}\label{eq2.1}
  g(u):=\frac{f(u_F)}{f(u)} 
\end{equation}
is well defined for all $u\neq r$ and we can rewrite the eigenvalue
equation $\Hc_p(f)(u) = \lambda \varrho_u \phi_p(f(u))$ as

\begin{equation}\label{eq3}
  \omega_{uu_F}\phi_p(1-g(u))
  =\lambda\varrho_u-\kappa_u
  -\sum_{v\prec u }\omega_{uv}\phi_p\Big(1-\frac{1}{g(v)}\Big)\,,
\end{equation}
for each $u\neq r$. 

Now, if $u$ is a leaf, Equation \eqref{eq3} allows us to write $g(u)$
explicitly as a function of $\lambda$:
\begin{equation}\label{eq4}
  g_u(\lambda)=1+\phi_p^{-1}\Big(
  \frac{\kappa_u-\varrho_u\lambda}{\omega_{uu_F}}\Big)\,.
\end{equation}
Similarly, for a generic node $u$ different from the root, we can use
\eqref{eq3} to characterize $g(u)$ implicitly as a function of the
variable $\lambda$:
\begin{equation}\label{eq5}
  g_u(\lambda)=1+\phi_p^{-1}\Bigg(\frac{\kappa_u-\varrho_u \lambda
    +\sum_{v\prec u }\omega_{uv}
    \phi_p\Big(1-\frac{1}{g_v(\lambda)}\Big)}{\omega_{uu_F}}\Bigg)\,.
\end{equation}
Finally, for the root $u=r$, we define
\begin{NEW}
\begin{equation}\label{eq6}
  g_r(\lambda):=1+\phi_p^{-1}\Big(\kappa_r-1-\lambda\varrho_r
  +\sum_{v\prec r }\omega_{rv}\phi_p\Big(1-\frac{1}{g_v(\lambda)}\Big)\Big)\,.
\end{equation}
We call the functions $g_u$ defined in \eqref{eq4}, \eqref{eq5}, \eqref{eq6} the \textit{generating functions} of the eigenfunction $f$. In fact,  we will show in Section \ref{sec:g_u_eigenfunctions} that $g_u(\lambda)$ characterizes the ratio $f(u_F)/f(u)$ for any eigenfunction of $\lambda$ such that $f(u)\neq 0$. To this end, we need a number of preliminary results to unveil several properties of the generating functions $g_u$. 

First, observe that when $f\neq 0$ everywhere, the claimed characterizing property follows directly from the definition of $g_u$. We highlight this statement in the following remark.

\begin{remark}\label{remark_tree_1}
If $\lambda$ is an eigenvalue of $\Hc_p^{u_0}$ for some $u_0\in V$, and $f$ is an associated eigenfunction  such that $f(u)\neq 0$, $\forall u\in\Tc_{u_0}$, then by the definition of the functions $g_u(\lambda)$ one directly obtains  that  
$$\frac{f(u_F)}{f(u)}=g_u(\lambda)\neq 0, \quad\forall u\in\Tc \setminus\{u_0\}\qquad \text{and} \qquad g_{u_0}(\lambda)=0\,. $$
\end{remark}
\end{NEW}
On the other hand, it is not difficult to observe that also the opposite property holds, namely
\begin{remark}\label{remark_tree_2}
Assume that $\lambda$ is a zero of $g_{u_0}(\lambda)$
and $g_{u}(\lambda)\neq 0$, for all $u<u_0$, i.e., for all the
descendents of $u_0$ and not only the direct children. Then $\lambda$
is an eigenvalue of $\Hc_p^{u_0}$ and a corresponding eigenfunction
$f$ can be defined on the subtree $\Tc_{u_0}$ by setting $f(u_0)=1$
and $f(u) = \frac{f(u_F)}{g_u(\lambda)}$, for all $u<u_0$.
Indeed, with these definitions, \eqref{eq4} and \eqref{eq5}
imply that $\lambda$ and $f$ are solutions of the system of
equations
\begin{align*}
 \begin{cases}
 \displaystyle
  \sum_{v\prec u_0}\omega_{u_0v}\phi_p(f(u_0)-f(v))
    +\omega_{u_0u_{0,f}}\phi_p(f(u_0))+\kappa_u\phi_p(f(u_0))
    =\lambda \varrho_u \phi_p(f(u_0)) \, , & \\[1.5em]
    \displaystyle
  \sum_{v\in \Gc}\omega_{uv}\phi_p(f(u)-f(v))
    +\kappa_u\phi_p(f(u))=\lambda \varrho_u \phi_p(f(u))
    \qquad \forall u<u_0 \, , &
\end{cases}   
\end{align*}
which shows that $\lambda$ and $f$ are an eigenvalue and an
eigenfunction of $\Hc_p^{u_0}$.
\end{remark}

\begin{NEW}
We have observed already that it is possible to relate the eigenpairs of the subtrees of $\Tc$ with the values of the functions $g_u(\lambda)$. Then, we show that it is always possible to immerse the tree $\Tc$ in a larger tree  for which the values of the functions $g_u(\lambda)$ do not  change. 

\begin{remark}\label{remark_tree_3}
Let $\Hc_p$ be the generalized $p$-Laplacian operator defined on a tree $\Tc = (V,E)$. We can always immerse $\Tc$ in a tree $\tilde{\Tc}$ obtained adding a parent $r_F$ to the root $r$. Next, we define the generalized $p$-Laplacian operator $\tilde{\Hc}_p$ on $\tilde{\Tc}$ by setting $\tilde{\omega}_{uv}=\omega_{uv}$, $\forall (u,v)\in E$, $\tilde{\omega}_{rr_F}=1$, $\tilde{\kappa}_u=\kappa_u$, $\forall u\in V \setminus\{r\}$ and $\tilde{\kappa}_r=\kappa_r-1$.
Considering $\tilde{\Tc}_{r_F}$ and $\tilde{\Hc}_p^{r_F}$, the subtree and the operator obtained removing the root $r_F$ from $\tilde{\Tc}$, it is straightforward to observe that $\Tc=\tilde{\Tc}_{r_F}$ and $\Hc_p=\tilde{\Hc}_p^{r_F}$.
Morover, working on $\tilde{\Tc}$ and the associated operator $\tilde{\Hc}_p$, it is possible to introduce the functions $\tilde{g}_u(\lambda)$ as in \eqref{eq4}, \eqref{eq5}, \eqref{eq6}.
$$g_u(\lambda)=\tilde{g}_u(\lambda), \quad \forall u\in\Tc\, .$$
Thus, the generalized $p$-Laplacian eigenavalue problem on a tree can  always be studied as the generalized $p$-Laplacian eigenvalue problem on a subtree of a suitable larger tree. \end{remark}
\end{NEW}
Finally, the following lemma summarizes several relevant structural properties of the functions $g_u(\lambda)$.
\begin{lemma}\label{Lemma3}
  For each $u\in  V$, consider the function $g_u(\lambda)$
  defined as in \eqref{eq4}--\eqref{eq6}. Then:
  \begin{enumerate}
  \item the poles of $g_u(\lambda)$ are the zeros of the functions
    $\lbrace g_v(\lambda)\rbrace_{v\prec u}$;
  \item $g_u$ is strictly decreasing between each two consecutive poles;
  \item $\lim_{\lambda\to -\infty}g_u=+ \infty $,
    $\lim_{\lambda\to +\infty}g_u=- \infty $,
    $\lim_{\lambda\to p^{-}}g_u=- \infty $,
    $\lim_{\lambda\to p^{+}}g_u=+ \infty $ where $p$ is any of the
    poles.
  \end{enumerate} 
\end{lemma}

\begin{proof}
  Let $u\in V$, if $u$ is a leaf then the three
  properties follow immediately from \eqref{eq4}. Otherwise, assume by
  induction the thesis holds for each $v\prec u$. From \eqref{eq5}, it
  immediately follows that the poles of $g_u$ are the zeros of
  $\lbrace g_v \rbrace_{v\prec u}$. To show that the function
  $\sum_{v\prec u}\omega_{uv}\phi_p\Big(1-\frac{1}{g_v(\lambda)}\Big)$
  is strictly decreasing between any couple of neighboring poles,
  observe that $x\mapsto\phi_p(x)$ is strictly increasing and, by
  induction, $\forall v\prec u$, $\lambda\mapsto g_v(\lambda)$ is
  strictly decreasing between any two of its zeros (i.e.\ the poles of
  $g_u$).  Moreover, since $\lambda\mapsto -\varrho_u \lambda$ is
  decreasing and $\phi_p^{-1}$ increasing, we can conclude that the
  function $\lambda\mapsto g_u(\lambda)$ is strictly decreasing
  between any two of its poles.  Finally, the limits of $g_u(\lambda)$
  for $\lambda \to p^\pm$ in the third statement follow as a
  consequence of the previous observations, while the limits for
  $\lambda\to\pm\infty$ can be proved directly by the induction
  assumption.
\end{proof}
\begin{NEW}
\subsection{Eigenfunction characterization via generating functions}\label{sec:g_u_eigenfunctions}
The following result shows that the generating functions $g_u(\lambda)$ always characterize the eigenfunctions of $\lambda$,  generalizing what observed earlier in Remark~\ref{remark_tree_1}.
\begin{theorem}\label{thm_prop_ratios_of_eigenfunctions}
Let $(f,\lambda)$ be an eigenpair of a generalized $p$-Laplacian operator, $\Hc_p$, defined on a tree $\Tc=(V,E)$ with root $r$. For any node $u\in V\setminus\{r\}$ such that $f(u)\neq 0$, it holds
$$\frac{f(u_F)}{f(u)}=g_u(\lambda)\, ,$$
where $u_F$ is the parent of $u$ in $\Tc$.
\end{theorem}
\begin{proof}
If $f(v)\neq 0\;\forall v\in\Tc$ we have already observed in Remark
\ref{remark_tree_1}  that the thesis holds. Assume thus that there exist $v_1,\dots,v_k\in V$ such that 
$$f(v_i)=0, \quad i=1,\dots,k\qquad \text{and} \qquad  f(u)\neq 0, \quad\forall u\not\in\{v_i\}_{i=1}^k$$
and let $\Tc'=\mathop{\sqcup}_{\substack{i=1}}^h \Tc_i$ and $\Hc_p'=\mathop{\oplus}_{i=1}^h\Hc_p(\Tc_i)$ be the forest and the corresponding operator obtained from $\Gc$ removing the nodes $v_1,\dots,v_k$ as in Section \ref{sec:remove-a-node}.
From Remark \ref{Remark_removing_a_node},  $\forall\, i=1,\dots,h$, the pair $(f|_{\Tc_i},\lambda)$ is an eigenpair of $\Hc_p(\Tc_i)$ such that $f|_{\Tc_i}(u)\neq 0,\;\forall\,u\in\Tc_i$. Denoting with $r_i$ the root of $\Tc_i$ and using \eqref{eq4},\eqref{eq5}, \eqref{eq6} and Remark \ref{remark_tree_1}, $\forall\, \Tc_i$, starting from the leaves, we can define functions $g_u^{\Tc_i}(\lambda)$ such that 
\begin{equation}
\begin{cases}
g_u^{\Tc_i}(\lambda)=\displaystyle{\frac{f|_{\Tc_i}(u_F)}{f|_{\Tc_i}(u)}}\neq 0 \qquad \forall u\in\Tc_i\setminus\{r_i\}\\[.7em]
g_{r_i}^{\Tc_i}(\lambda)=0 
\end{cases}
\end{equation}
We claim that $\forall i=1,\dots,h$ and $\forall u\in\Tc_i$, then $g_u(\lambda)=g_u^{\Tc_i}(\lambda)$. The thesis follows directly from this claim since
\begin{equation*}
g_u(\lambda)=g_u^{\Tc_i}(\lambda)=\frac{f(u_F)}{f(u)}\qquad\forall\, u\in\: \Tc_i\,.
\end{equation*}

To prove the claim, first we introduce a partial ordering on $\{\Tc_i\}_{i=1}^h$ and $\{v_j\}_{j=1}^k$ so that $\Tc_i\prec v_j$ if $v_j$ is the parent of the root of $\Tc_i$, while $v_j\prec \Tc_i$ if $v_j$ is the child of some node of $\Tc_i$.
Then, if $v_j\prec \Tc_i$ there exists a subtree $\Tc_l\prec v_j$. In fact, considering the generalized $p$-Laplacian eigenvalue equation in $v_j$ with $u_i={v_j}_F\in\Tc_i$,  we can write
$$\omega_{v_ju_i}\phi_p\Big(f(u_i)\Big)+\sum_{u\prec~v_j}\omega_{v_ju}\phi_p\Big(f(u)\Big)=0\,.$$
Since $f(u_i)\neq 0$, there exists a node $u_l\prec v_j$ such that $f(u_l)\neq0$ i.e. $u_l\in\Tc_l\prec v_j$. Similarly, one observes that if $f(v_j)=0$, and $v_j$ is a leaf, then also $f({v_j}_F)=0$. Because of these two facts, there exists some $\Tc_{i_0}$ in the set  $\{\Tc_i\}_{i=1}^h$ such that a node $v_j$ with $v_j\prec \Tc_{i_0}$ cannot exists. In addition, the leaves of $\Tc_{i_0}$ are all and only the leaves of $\Tc$ that are connected to $\Tc_{i_0}$. 
It is then easy to observe that, for any such $\Tc_{i_0}$, by definition, $g_u^{\Tc_{i_0}}(\lambda)=g_u(\lambda)\;\forall\,u\in \Tc_{i_0}$, $u\neq r_{i_0}$.
Moreover, when $u=r_{i_0}\prec v_j$ we have
\begin{equation}\label{ratios_eq_1}
  \begin{aligned}
    g_{r_{i_0}}^{\Tc_{i_0}}(\lambda):&=1+\phi_p^{-1}\Bigg(\kappa_{r_{i_0}}'-1-\lambda\varrho_{r_{i_0}} +\sum_{v\prec {r_1} }\omega_{r_{i_0} v}\phi_p\Big(1-\frac{1}{g_v(\lambda)}\Big)\Bigg)=0
  \end{aligned}\,,
\end{equation}
which implies 
\begin{equation}\label{ratios_eq_1.2}
   \kappa_{r_{i_0}}+\omega_{r_{i_0} v_j}-\lambda\varrho_{r_{i_0}} +\sum_{v\prec {r_1} }\omega_{r_{i_0} v}\phi_p\Big(1-\frac{1}{g_v(\lambda)}\Big)=0\,,
\end{equation}
where, since $v_j$ is one of the removed nodes, we have used the expression $\kappa_{r_{i_0}}'=\kappa_{r_{i_0}}+\omega_{r_{i_0}v_j}$ that we obtain when moving from $\Hc_p$ to $\Hc_p'$ as in Section \ref{sec:remove-a-node}. 

Thus, \eqref{ratios_eq_1.2} implies 
\begin{equation}\label{ratios_eq_2}
\begin{aligned}
g_{r_{i_0}}(\lambda)&=1+\phi_p^{-1}\Bigg(\frac{\kappa_{r_{i_0}}-\varrho_{r_{i_0}} \lambda_0 +\sum_{v\prec r_{i_0} }\omega_{r_{i_0} v} \phi_p\Big(1-\frac{1}{g_v(\lambda)}\Big)}{\omega_{r_{i_0}v_j}}\Bigg)\\&=1+\phi_p^{-1}\Bigg(\frac{-\omega_{r_{i_0}v_j}}{\omega_{r_{i_0}v_j}}\Bigg)=0\,,
\end{aligned}
\end{equation}
that is $g_{r_{i_0}}(\lambda)=g_{r_{i_0}}^{\Tc_{i_0}}(\lambda)=0$ and $\lambda$ is a pole of $g_{v_j}$, due to Lemma \ref{Lemma3}.

Now, given a general subtree $\Tc_{i_0}$, w.l.o.g. we can assume that the claim is true for any $\Tc_i\prec v_j\prec\Tc_{i_0}$.
Then if $u$ is a leaf of $\Tc_{i_0}$ that is also a leaf of $\Tc$, clearly 
$$g_u^{\Tc_{i_0}}(\lambda)=g_u(\lambda)\,.$$
Consider now the case of a leaf, $u$, of $\Tc_{i_0}$ that is not a leaf of $\Tc$. 
Since $u$ is not a leaf of $\Tc$, by construction, there exist some node $v_j\prec u$ and some subtree $\Tc_i\prec v\prec \Tc_{i_0}$. For any such $v_j$, by the inductive assumption, $\lambda$ has to be a pole of the corresponding $g_{v_j}$, leading to the following equation:
\begin{equation}
\begin{aligned}
g_u(\lambda)&=1+\phi_p^{-1}\Bigg(\frac{\kappa_u-\varrho_u \lambda +\sum_{v_j\prec u }\omega_{uv_j} \phi_p\Big(1-\frac{1}{g_{v_j}(\lambda)}\Big)}{\omega_{uu_F}}\Bigg)\\
&=1+\phi_p^{-1}\Bigg(\frac{\kappa_u-\varrho_u \lambda +\sum_{v_i\prec u }\omega_{uv_i} \phi_p(1)}{\omega_{uu_F}}\Bigg)\\
&=1+\phi_p^{-1}\Bigg(\frac{\kappa_u'-\varrho_u\lambda}{\omega_{uu_F}}\Bigg)=g_u^{\Tc_i}(\lambda)\,.
\end{aligned}
\end{equation} 
Here we have used as before the fact $\kappa_u'=\kappa_u+\sum_{v_j\prec u} \omega_{uv_j}$, see Section \ref{sec:remove-a-node}.

The case of $u$ a generic node of $\Tc_{i_0}$ can be proved analogously assuming, w.l.o.g., the claim true for any $w<u$, $w\in\Tc_{i_0}$.
Indeed, recalling $\kappa_u'=\kappa_u+\sum_{v_j\prec u} \omega_{uv_j}$ and that, by the inductive assumption, $\lambda$ is a pole of $g_{v_j}$ for any $v_j\prec u$, we get
\begin{equation}
\begin{aligned}
g_u(\lambda)&=1+\phi_p^{-1}\Bigg(\omega_{uu_F}^{-1}\bigg(\kappa_u-\varrho_u \lambda_0 +\displaystyle{\sum_{v_j\prec u }}\omega_{uv_j}+\displaystyle{\sum_{\substack{w\prec u\\w\in\Tc_{i_0}}}}\omega_{uw}\phi_p\Big(1-\frac{1}{g_{w}(\lambda)}\Big)\bigg)\Bigg)
\\
&=1+\phi_p^{-1}\Bigg(\omega_{uu_F}^{-1}\bigg(\kappa_u'-\varrho_u\lambda+\displaystyle{\sum_{\substack{w\prec u\\w\in\Tc_{i_0}}}}\omega_{uw}\phi_p\Big(1-\frac{1}{g_{w}(\lambda)}\Big)\bigg)\Bigg)=g_u^{\Tc_i}(\lambda)
\end{aligned}
\end{equation} 
The case of $u=r_{i_0}$ can be finally dealt with as done in  \eqref{ratios_eq_1}\eqref{ratios_eq_2}, concluding the proof. 
\end{proof}

\begin{corollary}\label{Corollary_zeros_in_zeros}
Let $(f,\lambda)$ be an eigenpair of $\Hc_p$, then, if $g_u(\lambda)=0$, necessarily $f(u_F)=0$.
\end{corollary}
\begin{proof}
  First, notice that Remark \ref{remark_tree_3} allows us to assume  that, given any $u\in\Tc\setminus{r}$, also the node $u_F$ has a parent, since we can always think of $\Tc$ as immersed in a larger tree with a suitably defined generalized $p$-Laplacian. Assume by contradiction that  $f(u_F)\neq0$, then by Theorem \ref{thm_prop_ratios_of_eigenfunctions} we would have that
\begin{equation}
\frac{f({u_F}_F)}{f(u_F)}=g_{u_F}(\lambda)\,.
\end{equation} 
At the same time, Lemma \ref{Lemma3} implies that  $\lambda$ is a pole of the function $g_{u_F}$, leading to a contradiction. 
\end{proof}

\subsection{Multiplicity via generating functions}
Theorem \ref{thm_prop_ratios_of_eigenfunctions} shows that given any eigenpair $(f,\lambda)$,  the generating functions $\{g_u(\lambda)\}_u$ characterize the value of $f$ up to a scaling factor. 
In this section we observe that counting the number of generating functions that vanishes on the eigenvalue $\lambda$ provides several insights about its multiplicity. 

First, we obtain the following sufficient result for simple eigenvalues, which directly follows from Theorem~\ref{thm_prop_ratios_of_eigenfunctions}. 
\begin{proposition}\label{Remark2}
  Let $\Hc_p$ be the generalized $p$-Laplacian operator defined on a tree $\Tc=(V,E)$, and let $u_0\in V$. If $g_u(\lambda)\neq 0$, $\forall u<u_0$ and $g_{u_0}(\lambda)=0$, then $\lambda$ is a simple eigenvalue of
  $\Hc_p^{u_0}$ associated to an everywhere
  nonzero eigenfunction.
\end{proposition}
\begin{proof}
 We have alredy observed in Remark \ref{remark_tree_2} that such a non zero eigenfunction $f$ exists. Assume by absurd that there exist also an eigenfunction $f^*$ of $\Hc_p^{u_0}$ associated to $\lambda$ with  $f^*\neq cf$,  $\forall c\in\R$. Then, due to Theorem \ref{thm_prop_ratios_of_eigenfunctions}, there has to exist a node $v$ such that $f^*(v)=0$. Since $f^*(v)=0$ and for any  node $u$ such that $f^*(u)\neq 0$ it holds that $f^*(u_F)=f^*(u) g_u(\lambda)\neq0$, then necessarily we get that $f^*(u)=0$,  $\forall u<v$.
On the other hand, by the generalized $p$-Laplacian eigenvalue equation, if $f^*(v)=0$ and $f^*(u)=0$, then $\forall u\prec v$ we have in addition that $f^*(u_F)=0$. Thus, if $f^*$ is zero in some node then necessarily $f^*=0$ everywhere, yielding a contradiction.
\end{proof}

Next, in the following lemma, we establish a more general condition for $\lambda$ to be an eigenvalue, counted with its $\gamma$-multiplicity, in terms of zeros of the generating functions $g_u(\lambda)$.

\begin{lemma}\label{lemma_multiplicity_trees}
Let $\Hc_p$ be a generalized $p$-Laplacian on a forest $\Gc$ and,
for any $u\in V$ and any tree of the forest, let $g_u(\lambda)$ be the function defined in \eqref{eq4}\eqref{eq5}\eqref{eq6}. Given $\lambda$, assume there exist $v_1,\dots,v_k\in V$ such that 
$$g_{v_i}(\lambda)=0\qquad \forall i=1,\dots,k\,.$$
Let $\{u_j\}_{j=1}^h$ be the set of the parents of the nodes $\{v_i\}_{i=1}^k$, where roots do not have parents. 
Then, $\lambda$ is an eigenvalue of $\Hc_p$ if and only if $k-h>0$,  and 
$$\gmult(\lambda)=k-h\, .$$
\end{lemma}
\begin{proof}
From Theorem \ref{thm_prop_ratios_of_eigenfunctions} and Corollary \ref{Corollary_zeros_in_zeros} we know that if $\lambda$ is an eigenvalue, any corresponding eigenfunction $f$ is such that 
$$f(u_j)=0\qquad \text{and} \qquad \frac{f(w_F)}{f(w)}=g_w(\lambda)\quad \mathrm{if}\; f(w)\neq 0 \,.$$  
Following the strategy of Section \ref{sec:remove-a-node}, remove the nodes $\{u_j\}_{j=1}^h$ from the forest $\Gc$ ending with a forest $\Gc'$ and an associated operator $\Hc_p'$ of the form
\begin{equation}
  \Gc'=\mathop{\sqcup}_{\substack{l=1}}^n \Tc_l \qquad \Hc_p'=\mathop{\oplus}_{l=1}^n \Hc_p(\Tc_l)\,.
\end{equation}
for some $n\geq 1$.
Then, from Remark \ref{Remark_removing_a_node}, any eigenfunction of $\Hc_p$ with eigenvalue $\lambda$ corresponds to an eigenfunction of $\Hc_p'$.
In particular, given any subtree $\Tc_l$ and corresponding operator $\Hc_p(\Tc_l)$ it is easy to observe that  
$$g_u^{\Tc_l}(\lambda)=g_u(\lambda)\qquad u\in\Tc_l\,,$$
where   $g_u^{\Tc_l}$ are the generating functions defined starting from $\Hc_p(\Tc_l)$ via equations \eqref{eq4},\eqref{eq5},\eqref{eq6} (see the proof of Theorem~\ref{thm_prop_ratios_of_eigenfunctions} for a similar construction).
 Among the $\{\Tc_l\}_{l=1}^n$, let  $\{\Tc_{i}'\}_{i=1}^k$ be the subtrees with root $r_i=v_i$. Due to Proposition \ref{Remark2}, for any such $\Tc_{i}'$ and corresponding $\Hc_p(\Tc_{i}')$ there exists a unique everywhere nonzero eigenfunction $f_{i}'$ of  $\Hc_p(\Tc_{i}')$ with eigenvalue $\lambda$ whose ratios $f_{i}'(w_F)/f_{i}'(w)$ are induced by the functions $g_w(\lambda)$, $\forall w\in\Tc_{i}'$. Moreover, notice that for any $f$ eigenfunction of $\Hc_p$ with eigenvalue $\lambda$, since $f$ is also an eigenfunction of $\Hc_p'$, we have $f|_{\Tc_i'}=\alpha_i f_i'$, for some $\alpha_i\in \R$.

On the other hand, on the subtrees $\{\Tc_{j}''\}_{j=1}^{n-k}$ whose root $r_j$ is such that $r_j\neq v_i$ $\forall i=1,\dots,k$, since $g_{w}(\lambda)\neq 0$, $\forall w\in\Tc_{j}''$, any eigenfunction associated to $\lambda$ of $\Hc_p$ has to be such that $f|_{\Tc_{j}''}(w)=0$ because of Theorem \ref{thm_prop_ratios_of_eigenfunctions}. Indeed, suppose by contradiction that $f$ is an eigenfunction associated to $\lambda$ such  that $f|_{\Tc_{j}''}\neq0$, then $f$  should be an eigenfunction of $\Hc_p(\Tc_{j}'')$ with same eigenvalue $\lambda$. However, $g_w(\lambda)\neq 0$, $\forall w\in \Tc_{j}''$ implies that $f|_{\Tc_{j}''}(w)\neq 0$, $\forall w\in \Tc_{j}''$ and thus, by Remark \ref{remark_tree_1}, we would have that $g_{r_j}(\lambda)=0$, which is absurd.

Now, let $f_i$ be the immersion of $f_i'$ into $\R^N$ such  that $f_i|_{\Tc_i'}=f_i'$ and $f_i(w)=0$ for all $w\notin \Tc_i'$. Define $\Omega:=\mathrm{span}\lbrace f_{i}\rbrace_{i=1}^k$  the $k$-dimensional linear space spanned by the $f_{i}$. The observations above together with Corollary \ref{Corollary_zeros_in_zeros} imply that if $f$ is an eigenfunction of $\Hc_p$ with eigenvalue $\lambda$, then $f\in \Omega$.
Starting from $\Omega$, we want to recover all the possible eigenfunctions of $\Hc_p$ relative to $\lambda$. To this end,  we select among the functions $f\in\Omega$ all those functions that satisfy the  eigenvalue equation for $\Hc_p$ also in the removed points $\{u_j\}_{j=1}^h$. For any node $u_j$, let $w_{i,j}$ be the node in the neighborhood of $u_j$  such that 
$w_{i,j}\in\Tc_{i}'$. Then, the $\Hc_p$ eigenvalue equation on a node $u_j$ reads
\begin{equation}\label{eq_lemma_mult_2}
\begin{aligned}
\Theta_{j}(f)&:=\sum_{i}\omega_{w_{i,j}u_{j}}\phi_p(\beta_{w_{i,j}})\phi_p(f_i(r_{i}))\\
&=\sum_{i}\omega_{w_{i,j}u_{j}}\phi_p(f_i(w_{i,j}))=\Big(\lambda \varrho_{u_j} -\kappa_{u_j}\Big)\phi_p(f(u_j)=0
&\quad \forall j=1,\dots,h
\end{aligned}
\end{equation}
where we have used the fact that on any $\Tc_{i}'$ the ratios between the components of $f_i$ are fixed by the functions $g_w(\lambda)$, $w\in\Tc_{i}'$ and thus, for every $w\in\Tc_{i}'$ there exists $\beta_w\neq 0$ such that $f_i(w)=\beta_w f_i(r_{i})$.

We continue by defining the set $A=\{f\,|\,\Theta_j(f)=0, \,\, \forall\,j=1,\dots,h\}$. It is clear that $f$ is an eigenfunction of $\Hc_p$ relative to $\lambda$ if and only if $f\in A\cap\Omega$. Thus, let us now study the genus of such a set.
Observe that $\gamma(A\setminus \{0\})=N-h$, since $A$ is diffeomorphic to a linear subspace of dimension $N-h$ through the homeomorphism of $\R^N$ given by $x_i\mapsto \phi_p(x_i)$, $i=1,\dots,N$ (the set of equations $\{\Theta_j(f)=0\}$ is transformed into a set of $h$ linearly independent equation by the change of variable $y_i:=\phi_p(f_i(r_i))$). 
Thus, if $k>h$ then the intersection is always nonempty because of Lemma \ref{Lemma1} and in particular
\begin{equation*}
  \gamma(A\cap\Omega\setminus\{0\})\geq \gamma(A\setminus\{0\})-(N-k)= N-h-N+k=k-h\,.
\end{equation*}

Now we claim that it is possible to define a function $\tilde{\psi}$ in the set of Krasnoselskii test maps $\Lambda_{k-h}(\Omega\,\cap\, A\setminus\{0\})$ such that $0\not\in \tilde{\psi}(\Omega\,\cap\, A\setminus\{0\})$. This implies $\gamma(\Omega\,\cap\, A\setminus\{0\})\leq k-h$, from which the statement follows.
To construct such $\tilde \psi$, consider the function $\psi\in\Lambda_k(\Omega)$ given by:
\begin{equation}\label{eq_lemma_mult_1}
\begin{aligned}
     f=\sum_{i=1}^k\alpha_i f_{i} \mapsto \psi(f):= \Big(f(r_{1})\,,\dots,f(r_{k})\Big)=\Big(\alpha_1 f_{1}(r_1),\dots,\alpha_k f_{k}(r_k)\Big)\,.
\end{aligned}
\end{equation}
It is easy to verify that $0\not\in\psi(\Omega\setminus\{0\})$, as $f_{i}(r_{i})\neq 0$, $\forall i$. 
Since we want to define the function $\tilde{\psi}$ on $A\cap \Omega$, we define $\tilde{\psi}$ as the restriction to $\R^{k-h}$ of $\psi$. 
To define such a restriction, note that among the $\{\Tc_i'\}$ it is possible to select $h$ distinct subtrees $\{\Tc_{i_l}'\}_{l=1}^h$ such that any node $u_j$ is incident to some $\Tc_{i_l}'$. As before, let $w_{i_l,j}$ be the neighbor of $u_j$ in $\Tc_{i_l}'$. 
Then consider the function $\tilde{\psi}:\Omega\,\cap\, A \to \R^{k-h}$, entrywise defined as 
\begin{equation*}
  \Big(\tilde{\psi}(f)\Big)_i=\Big(\psi(f)\Big)_i \qquad i\neq i_l,\;\;l=1,\dots,h\,.
\end{equation*}
It is easily proved that $\tilde\psi\in \Lambda_{k-h}(\Omega\,\cap\, A)$. Finally, we show that  if $\tilde{\psi}(f)=0$, for some $f\in \Omega\cap\, A$, then necessarily $f=0$. To this end, write $f=\sum_{i=1}^k \alpha_i f_{i}$. If  $\tilde{\psi}(f)=0$, then  (up to a reordering of the indices of the chosen subtrees)
\begin{equation*}
  \alpha_i f_i(r_i)=0 \qquad \forall i\neq i_l, \, l=1\dots,h\, .
\end{equation*}
Thus,  $f=\sum_{l=1}^h\alpha_{i_l} f_{i_l}$.
Then, observe that, since $f\in A$, we have
\begin{equation}\label{eq_20}
\Theta_j(f)=\sum_{l} \omega_{w_{i_l,j}u_j}\phi_p(\alpha_{i_l})\phi_p(\beta_{w_{i_l,j}}) \phi_p(f_{i_l}(r_{i_l}))=0\quad \forall j=1,\dots,h\,.
\end{equation}
Consider a node $u_{j_0}$ that is incident only to one of the subtrees $\{\Tc_{i_l}'\}_{l=1}^{h}$, say $\Tc_{i_{h}}'$, observe that such a node necessarily exists because there are no loops in a forest.  
Then \eqref{eq_20} for $j=j_0$ reads 
(up to a reordering of the indices)
\begin{equation}\label{eq_21}
  \Theta_{j_0}(f) = \omega_{u_{j_0}w_{i_{h},j_{0}} } \phi_p(\alpha_{i_{h}})\phi_p(\beta_{w_{i_{h}}})\phi_p(f_{i_{h}}(r_{i_{h}}))=0\,.
\end{equation}
This means that $\alpha_{i_{h}}=0$, i.e.
$f=\sum_{l=1}^{h-1} \alpha_{i_l} f_{i_l}$. Repeating this procedure for all the $h$ nodes $u_j$, we obtain that all the $\alpha_i$ have to be zero. In particular, this implies that, if $k=h$, then all the $\alpha_i$ are zero and thus $A\cap \Omega=\{0\}$ i.e.\ $\lambda$ is not an eigenvalue, thus concluding the proof.
\end{proof}

To conclude this preparatory section needed to tackle the proofs of Theorems ~\ref{thm:main1-variational-eig-tree} and~\ref{thm:main2-nodal-count-trees}, we show in the next result how the eigenvalues and the corresponding $\gamma$-{mul\-ti\-pli\-ci\-ties}
change when moving  from $\Hc_p^u$ to $\Hc_p^{\tilde{u}}$ (recall that $\Hc_p^u$ is the operator obtained by removing all the nodes different from $u$ and its descendants while $\Hc_p^{\tilde{u}}$ is the one obtained by removing also the node $u$). We state this result as a corollary of the previous lemma, recalling that $\lambda$ is not an eigenvalue if and only if $\gmult(\lambda)=0$. 

\begin{corollary}\label{Corollary_multiplicities}$\phantom{.}$
  \begin{enumerate}
  \item Let $\lambda$ be such that $g_u(\lambda)=0$, then $\lambda$ is an eigenvalue of $\Hc_p^u$ and $\gmult_{\Hc_p^u}(\lambda)=\gmult_{\Hc_p^{\tilde{u}}}(\lambda)+1$. 
  \item Let $\lambda$ be an eigenvalue of $\Hc_p^{\tilde{u}}$ such that $g_w(\lambda)\neq 0$ for all $w\prec u$ and $g_u(\lambda)\neq 0$, then $\lambda$ is an eigenvalue of $\Hc_p^u$ such that 
    $\gmult_{\Hc_p^u}(\lambda)=\gmult_{\Hc_p^{\tilde{u}}}(\lambda)$. 
  \item Let $\lambda$ be an eigenvalue of $\Hc_p^{\tilde{u}}$ and assume there exist $w_1,\dots,w_h\prec u$ with $g_{w_i}(\lambda)=0$, then  $\gmult_{\Hc_p^u}(\lambda)=\gmult_{\Hc_p^{\tilde{u}}}(\lambda)-1$.
  \end{enumerate}
\end{corollary}
\begin{proof}
From Lemma \ref{lemma_multiplicity_trees} we know that $\lambda$ is an eigenvalue if and only if $k-h>0$, where $k$ is the number of nodes $v$ such that $g_v(\lambda)=0$ and $h$ is the number of their parents. In particular, we have that $\gmult_{\Hc_p}(\lambda)=k-h$.
To prove point 1, we observe that, since $u$ is the root of the subtree $\Tc_u$, $u$ has no parents and thus necessarily $h<k$. Moreover, by Lemma \ref{Lemma3}, $g_v(\lambda)\neq 0$, for all $v\prec u$ implying that $h$ does not change when moving from $\Tc_{\tilde{u}}$ to $\Tc_u$, while $k$ increases by one. This  implies the statement.
To prove point 2, it is enough to observe that the number $k-h$ does not change going from $\Tc_{\tilde{u}}$ to $\Tc_{u}$.
Finally, in order to prove point 3 observe that in this case, when moving from $\Tc_{\tilde{u}}$ to $\Tc_{u}$, $k$ does not change while $h$ increases by one.
\end{proof}
\end{NEW}
\begin{NEW}

\subsection{Proofs of Theorems ~\ref{thm:main1-variational-eig-tree} and \ref{thm:main2-nodal-count-trees}}

We are finally ready to prove the main Theorems~\ref{thm:main1-variational-eig-tree} and \ref{thm:main2-nodal-count-trees}.

\begin{proof}[Proof of Theorem \ref{thm:main1-variational-eig-tree}]

  We first observe that if the thesis holds for trees, then it holds as well for forests. To prove this fact, we assume that all the eigenvalues on trees are variational and the multiplicity matches the $\gamma$-multiplicity. Then, we note that if $\Gc=\sqcup_i \Tc_i$, with $\Tc_i=(V_i,E_i)$ trees,
  then $\Hc_p=\oplus_i{\Hc_p}(\Tc_i)$, where ${\Hc_p}(\Tc_i)$ is a
  suitable generalized $p$-Laplacian operator defined on $\Tc_i$, and
  hence $\sigma(\Hc_p)=\cup_i\sigma({\Hc_p}(\Tc_i))$. In other words, the spectrum of $\Hc_p$ is the union of the spectra of the operators defined on the trees forming $\Gc$. Next, we observe that, by the same assumption on trees,  $\sigma(\Hc_p(\Tc_i))$ is formed only by variational eigenvalues and thus it
  contains at most $|V_i|$ distinct  elements, implying that $\sigma(\Hc_p)$ is formed by at most $N$ different eigenvalues. 
Now, let $\lambda\in \sigma(\Hc_p)$. By the previous assumption, $\lambda$ is 
a variational eigenvalue of $\Hc_p(\Tc_i)$, for some  $i\in\{1,\dots,k\}$ and 
$\mult_{\Hc_p(\Tc_i)}(\lambda)=\gmult_{\Hc_p(\Tc_i)}(\lambda) = m_i(\lambda)$. Then, for any $i\in\{1,\dots,k\}$ there exists $\varphi_i\in\Lambda_{m_i(\lambda)}(A_{\lambda}^i)$ s.t. $0\not\in \varphi_i(A_{\lambda}^i\cap \mathcal S_p)$, where 
\begin{equation*}
A_{\lambda}^{i}=\{f:V_i\rightarrow \R\,|\,  \Hc_p(\Tc_i)(f)=\lambda |f|^{p-2}f\}\,.
\end{equation*}
Let
\begin{equation*}
  A_\lambda = \{f:V\rightarrow \R\,| \Hc_p(f)=\lambda |f|^{p-2}f\}\, .
\end{equation*}
Then, we can consider the extensions of the functions $\varphi_i$ to $A_\lambda$ and, given $m(\lambda) = \sum_i m_i(\lambda)$, define the function $\varphi_\lambda\in \Lambda_{m(\lambda)}(A_\lambda)$ as a linear combination of $\varphi_i$ such that $0\notin \varphi_\lambda(A_\lambda\cap \mathcal S_p)$. This implies that $\gmult_{\Hc_p}(\lambda)\leq m(\lambda)$.
Noting that $N=\sum_{\lambda}\sum_{i} m_i(\lambda)$, we have $\sum_{\lambda} \gmult_{\Hc_p}(\lambda)\leq N$. 
Thus, by Corollary \ref{Corollary_gmul_greater_than_N} we  conclude that all the eigenvalues of $\Hc_p$ are variational  and $\mult_{\Hc_p}(\lambda)=\gmult_{\Hc_p}(\lambda)=\sum_{i=1}^k \mult_{\Hc_p(\Tc_i)}(\lambda)$.

Now, we address the proof of the assumption and consider the case in which $\Gc=\Tc$ is a tree.  The proof proceeds by induction on the number of nodes $N$. If $N=1$, from \eqref{eq4} and Proposition \ref{Remark2}, we can conclude that there exists only one eigenvalue, $\lambda_1$, with $\gmult_{\Hc_p}(\lambda_1)=\mult_{\Hc_p}(\lambda_1)=1$. Assume now that $N>1$ and that the theorem holds up to $N-1$. First note that the inductive assumption and the result derived in the previous paragraph imply that the thesis holds for any forest composed by trees, each one, with less then $N$ nodes. Then, fix a root $r$ for $\Tc$ and consider $ \Tc_{\tilde{r}}=\sqcup_{i=1}^n \Tc_i$ and $\Hc_p^{\tilde{r}}=\oplus_i \Hc_p(\Tc_i)$. Proceed by dividing the eigenvalues of  $\Hc_p^{\tilde{r}}$ into two sets $\{\nu_j\}_{j=1}^k$ and $\{\xi_l\}_{l=1}^h$, where each $\nu_j$ is a  zero of some $g_v$ for some $v\prec r$, whereas $\xi_l$  is not.
By the inductive assumption, we have that 
\begin{equation*}
  \sum_{j=1}^k \gmult_{\Hc_p^{\tilde{r}}}(\nu_j)+\sum_{l=1}^h \gmult_{\Hc_p^{\tilde{r}}}(\xi_l)=N-1\, .
\end{equation*}
Now, let us divide the eigenvalues of $\Hc_p$ in a similar way. Let $\{\mu_i\}_{i=1}^{k+1}$ the eigenvalues that are zeros of $g_r$.  By Lemma \ref{Lemma3} we know that they are exactly $k+1$. Lemma \ref{lemma_multiplicity_trees} ensures that all the other eigenvalues of $\Hc_p$ are also eigenvalues of $\Hc_p^{\tilde{r}}$ and, in particular, they must be either in the set $\{\nu_j\}_{j=1}^{k}$ or in the set $\{\xi\}_{l=1}^{h}$. Moreover, from Lemma\ref{Lemma3} we deduce that $\{\nu_j\}_{j=1}^{k}\cap\{\mu_i\}_{i=1}^{k+1}=\emptyset$ while $\{\xi_l\}_{l=1}^{h}\cap\{\mu_i\}_{i=1}^{k+1}$ could be non empty. In particular, let us set 
\begin{equation*}
\{\xi\}_{l=1}^{h_1}=\{\xi_l\}_{l=1}^{h}\setminus\{\xi_l\}_{l=1}^{h}\cap\{\mu_i\}_{i=1}^{k+1}\,.
\end{equation*}
Then, 
\begin{equation}
  \begin{aligned}
    &\sum_{i=1}^{k+1}\gmult_{\Hc_p}(\mu_i)+\sum_{j=1}^{k}\gmult_{\Hc_p}(\nu_j)+\sum_{l=1}^{h_1}\gmult_{\Hc_p}(\xi_l)\\
    =&\sum_{i=1}^{k+1}\Big(\gmult_{\Hc_p^{\tilde{r}}}(\mu_i)+1\Big)+\sum_{j=1}^{k}\Big(\gmult_{\Hc_p^{\tilde{r}}}(\nu_j)-1\Big)+\sum_{l=1}^{h_1}\gmult_{\Hc_p^{\tilde{r}}}(\xi_l)\\
    =&k+1-k+\sum_{j=1}^{k}\gmult_{\Hc_p^{\tilde{r}}}(\nu_j)+\sum_{l=1}^{h}\gmult_{\Hc_p^{\tilde{r}}}(\xi_l)=N-1+1=N
  \end{aligned}
\end{equation}
where we have used Corollary \ref{Corollary_multiplicities} and the fact that $\{\xi_l\}_{l=1}^{h}\subseteq(\{\xi_l\}_{l=1}^{h_1}\cup\{\mu_i\}_{i=1}^{k+1})$, with $\gmult_{\Hc_p^{\tilde{r}}}(\mu_i)=0$ if $\mu_i\not\in \{\xi_l\}_{l=1}^{h}$. Together with Corollary\ref{Corollary_gmul_greater_than_N}, the latter equality concludes the proof.
\end{proof}

Before moving on to the proof of Theorem~\ref{thm:main2-nodal-count-trees}, several observations are in order.

\begin{remark}\label{remark_multiplicity_on_forests_of_everywhere_non_zero_eigenfunctions}
Suppose $\Gc=\sqcup_{i=1}^m \Tc_i$ is a forest and let $\Hc_p=\mathop{\oplus}_{i=1}^m\Hc_p(\Tc_i)$ as before. If we consider an eigenfunction $f_k$ of $\Hc_p$ that is everywhere nonzero Theorem \ref{thm:main1-variational-eig-tree} ensures that the corresponding eigenvalue $\lambda_k$ has multiplicity exactly equal to $m$.  Indeed, necessarily $f_k|_{\Tc_i}$ is an eigenfunction of $\Hc_p(\Tc_i)$ and, since it is everywhere non-zero, its corresponding eigenvalue is simple because of Proposition \ref{Remark2}.
\end{remark}

\end{NEW}

In addition, observe that $e_0=(u_0,v_0)$ is an edge of some $\Tc_i$ such that $f_k(u_0)f_k(v_0)<0$ if and only if $e_0$ separates two distinct nodal domains. This means that the number of nodal domains induced by $f_k$ on $\Tc_i$ is equal to the number of edges where $f_k|_{\Tc_i}$ changes sign, plus one. Thus the total number of nodal domains induced by $f_k$ on $\Gc$ is equal to $m$ plus the total number of edges where $f_k$ changes sign.  Combining all these observations, we eventually obtain the following proof.

\begin{proof}[Proof of Theorem \ref{thm:main2-nodal-count-trees}]
  We prove by induction on $k$ that, if $f_k$ is an eigenfunction
  everywhere nonzero associated to the multiple eigenvalue
  $\lambda_k=\dots=\lambda_{k+m-1}$, then $f_k$ changes sign on
  exactly $k-1$ edges, implying that $f_k$ induces exactly $k-1+m$
  nodal domains.
  If $k=1$, $f_1|_{\Tc_i}$ is an eigenfunction related to the first
  eigenvalue of each operator $\Hc_p(\Tc_i)$, $i=1,\dots, m$. Thus, as
  a consequence of Theorem \ref{Theorem0.1}, $f_1$ is strictly
  positive or strictly negative on every tree $\Tc_i$ and overall it
  induces $m$ nodal domains. Moreover, it does not change sign on any
  edge.  Now we assume the statement to be true for every $h<k$ and
  prove it for $h=k$.  If $k>1$, then $f_k$ cannot be a first
  eigenfunction on every tree $\Tc_i$. Then, by
  Theorem~\ref{Theorem0.1}, there exists at least one edge
  $e_0=(u_0,v_0)$ in some $\Tc_{i_0}$ such that
  $f_k(u_0)f_k(v_0)<0$. Thus, we operate as in Section
  \ref{sec:remove-edge} and remove edge $e_0$ to disconnect
  $\Tc_{i_0}$ into the two subtrees $\Tc_{i_0}'$ and $\Tc_{i_0}''$, so
  that the reduced graph $\Gc'$ is the union of the $m+1$ subtrees:
  \begin{equation*}
    \Gc'=\Big(\mathop{\sqcup}_{\substack{i=1\\ i\neq i_0}}^m \Tc_i\Big)    \sqcup \Tc_{i_0}' \sqcup \Tc_{i_0}'' \, .
  \end{equation*}
  Similarly, the new operator $\Hc_p'$, obtained after removing $e_0$,
  can be decomposed as:
  \begin{equation*}
    \Hc'_p=\Big(\mathop{\oplus}_{\substack{i=1 \\ i\neq i_0}}^m
    \Hc_p(\Tc_i)\Big)\oplus \Hc'_p(\Tc_{i_0}')\oplus
    \Hc'_p(\Tc_{i_0}'')\, .
  \end{equation*}
  Now we can compare the eigenvalues $\lbrace \eta_k \rbrace$ of
  $\Hc_p'$ with the ones $\lbrace \lambda_k \rbrace$ of $\Hc_p$. From
  Lemma~\ref{Lemma1.2} we have:
  \begin{equation*}
    \eta_{k-1}\leq\lambda_k\leq\eta_{k}
    \leq \dots \leq \lambda_{k+m-1}\leq
    \eta_{k+m-1}\leq\lambda_{k+m}\,.
  \end{equation*}
  \new{Due to Remark \ref{remark_multiplicity_on_forests_of_everywhere_non_zero_eigenfunctions} and Theorem \ref{thm:main1-variational-eig-tree}, the multiplicity of $\eta_{k}$ has to
  be exactly $m+1$ and, by
  assumption,
  $\lambda_{k-1}<\lambda_k=\dots=\lambda_{k+m-1}<\lambda_{k+m}$, i.e. $\eta_{k-1}=\dots=\eta_{k+m-1}$.} 
  Moreover, by the inductive assumption, $f_k$ changes sign on $k-1$
  edges of the graph $\Gc'$. Thus, on the original graph $\Gc$, $f_k$
  changes sign $k-1+1=k$ times, concluding the proof.
\end{proof}

\section{Nodal domain count on generic graphs
  (proofs of Theorems \ref{thm:main3-nodal-count-simple}
  and \ref{thm:main4-nodal-count-general})}\label{sec:generic-graphs}

In this final section we prove Theorems
\ref{thm:main3-nodal-count-simple} and
\ref{thm:main4-nodal-count-general}, providing upper and lower bounds
for the number of nodal domains of $f$ and for the number of edges
where $f$ changes sign. To this end, we need first a few preliminary
results.

Consider an eigenpair $(f,\lambda)$ of the generalized $p$-Laplacian
operator $\Hc_p$ on a generic graph $\Gc$. Suppose we remove from
$\Gc$ an edge $e_0=(u_0,v_0)$, obtaining the graph
$\Gc' = \Gc\setminus\{e_0\}$. Modifying accordingly the generalized
$p$-Laplacian $\Hc_p$ as in Section \ref{sec:remove-edge}, the new
operator $\Hc_p'$ on $\Gc'$ is such that the pair $(f,\lambda)$,
restricted to $\Gc'$, remains an eigenpair of $\Hc_p'$.

Let us denote by $\Delta l(e_0,f)$ the variation of the number of
independent loops of constant sign, namely the difference between the
number of loops of constant sign of $f$ in $\Gc'$ minus the number of
those in $\Gc$. Similarly, let $\Delta \nu(e_0,f)$ be the variation
between the number of nodal domains induced by $f$ on $\Gc'$ and on
$\Gc$. We can characterize the difference
$\Delta \nu(e_0,f) - \Delta l(e_0,f)$ in terms of
$\mathrm{sign}_{e_0}(f)= f(u_0) f(v_0)$, i.e.\ whether or not $f$
changes sign over $e_0$. In fact, note that, by definition, if
$\mathrm{sign}_{e_0}(f)<0$, then neither the number of loops of
constant sign nor the number of nodal domains changes. If,
instead, $\mathrm{sign}_{e_0}(f)>0$, then either the number of
independent loops decreases by one ($\Delta(e_0,f)=-1$) or the number
of nodal domains increases by one ($\Delta(e_0,f)=+1$).  Overall, we
have
\begin{equation}\label{eq:Deltas}
  \Delta \nu(e_0,f)- \Delta l(e_0,f)=\begin{cases}
    0 \quad \mathrm{sign}_{e_0}(f)<0\\
    1 \quad \mathrm{sign}_{e_0}(f)>0
  \end{cases}\, .
\end{equation}
Based on the above formula, the following lemma provides a relation
between the number of nodal domains induced by an eigenfunction and
the number of edges where the sign changes. It is a generalization of
a result from \cite{Band}, which was proved for linear Laplacians and
for the case of everywhere nonzero functions.
\begin{lemma}\label{Lemma5} 
  \new{Consider $f:V\rightarrow \R$, a function on the graph $\Gc$}. Denote by $\zeta(f)$ the number of
  edges where $f$ changes sign, by $z(f)$ he number of nodes where
  $f$ is zero, by $l(f)$ the number of independent loops in $\Gc$
  where $f$ has constant sign, and by $|E_z|$ the number of edges
  incident to the zero nodes. Then
  \begin{equation*}
    \zeta(f)=|E|-|E_z|+z(f)-|V|+\new{\nu(f)-l(f)}\leq |E|-|V|+\new{\nu(f)-l(f)}\,.
  \end{equation*}
\end{lemma}

\begin{proof}
  Operating as in Section \ref{sec:remove-a-node}, we start by
  removing from $\Gc$ all the $z(f)$ nodes where $f$ is zero, thus
  obtaining a new graph $\Gc'$ with the corresponding new generalized
  $p$-Laplacian $\Hc_p'$. Since $|E_z|$ is the number of edges
  incident to the zero nodes that have been removed, the number of
  edges in $\Gc'$ can be estimated as $|E'|=|E|-|E_z|\leq
  |E|-z(f)$. Moreover, the edges incident to the zero nodes neither
  connect different nodal domains nor belong to constant sign
  loops. Hence, the restriction of $f$ to $\Gc'$ remains an
  eigenfunction of $\Hc_p'$ having the same number of nodal domains
  and the same number of constant sign loops as $f$.

  Next, we proceed by removing from $\Gc'$ all the edges
  $e_1,\dots,e_{\tau(f)}$ where $f$ does not change sign (i.e., such that
  $\mathrm{sign}_{e_i}(f)>0$) and modify consequently the operator
  $\Hc'_p$ as in Section \ref{sec:remove-edge}. We obtain a new graph
  $\Gc''$ and the corresponding new operator $\Hc''_p$ in such a way
  that $(\lambda,f)$ remains an eigenpair.  Since the
    number of edges where $f$ does not change sign is
  $\tau(f)=|E'|-\zeta(f)$, thanks to \eqref{eq:Deltas}, we have that
  \begin{equation}\label{zeros}
    \sum_{i=1}^{h}\Big(\Delta \nu(e_i,f)- \Delta l(e_i,f)\Big)
    =|E'|-\zeta(f)\leq|E|-z(f)-\zeta(f)\,.
  \end{equation}
  In the final graph $\Gc''$, only edges connecting nodes of different
  sign are present, so that each node is a nodal domain of $f$. As a
  consequence, there are a total $|V|-z$ nodal domains of $f$ and no
  loops with constant sign.  Then:
  \begin{equation*}
    \sum_{i=1}^{h}\Delta \nu(e_i,f)=|V|-z(f)-\nu(f)
    \qquad \text{and} \qquad  
    \sum_{i=1}^{h}\Delta l(e_i,f)=-l(f)\, ,
  \end{equation*}
  and, by~\eqref{eq:Deltas},
  $\tau(f)=\sum_{i=1}^{h}\Delta\nu(e_i,f)-\Delta~l(e_i,f)=|V|-z(f)-\nu(f)+l(f)$.
  Hence, using \eqref{zeros} and the fact that $\tau(f)=|E|-|E_z|-\zeta(f)$, we
  obtain
  \begin{equation*}
    \zeta(f)=|E|-|E_z|+z(f)-|V|+\nu(f)-l(f)\leq |E|-|V|+\nu(f)-l(f)
  \end{equation*}
  thus concluding the proof.
  \end{proof}

The above lemma allows us to prove our third and fourth main results
given in Theorems \ref{thm:main3-nodal-count-simple} and
\ref{thm:main4-nodal-count-general}, which provide new upper and lower
bounds for the number of nodal domains of the eigenfunctions of the
generalized $p$-Laplacian, extending and generalizing previous results
for the standard $p$-Laplacian and the linear Schr\"odinger operators
\cite{Tudisco1,Berkolaiko2,Xu}. As the proof of the two claims in
Theorem \ref{thm:main4-nodal-count-general} requires different
arguments, we subdivide it into two parts, each addressing one of the
two points P1 and P2 in the statement.

\begin{proof}[Proof of Theorem \ref{thm:main3-nodal-count-simple}]
  For a connected graph $\Gc$, let $f$ be an eigenfunction of $\Hc_p$
  relative to $\lambda$. Let $\nu(f)$ denote the number of nodal
  domains of $f$ and let $\Gc_1,\dots,\Gc_{\nu(f)}$ be such domains.
  Furthermore, let $e_1,\dots,e_{\zeta}$ be the edges where $f$
  changes sign and $v_1,\dots,v_z$ the nodes where $f$ is zero, \new{with $z=z(f)$ the number of such nodes}.  The
  proof proceeds as follows.

  According to Section~\ref{sec:remove-a-node}, we start by removing
  the nodes $v_1,\dots,v_z$ from $\Gc$ obtaining a new graph $\Gc'$.
  Operator $\Hc_p$ is then modified to form the operator $\Hc'_p$ in
  such a way that the restriction of $f$ to $\Gc'$ is an eigenfunction
  of $\Hc'_p$ with the same eigenvalue $\lambda$. Moreover, as all the
  zero nodes that are not part of any nodal domain are now removed, we
  observe that $f$ restricted to $\Gc'$ has no zeros and induces the
  same nodal domains that $f$ induces on $\Gc$.  From Lemma
  \ref{Lemma2}, we conclude that $\lambda<\lambda_k\leq\lambda_k'$,
  where $\lambda'_k$ denotes the $k$-th variational eigenvalue of
  $\Hc'_p$.

  Then, operating as in Section \ref{sec:remove-edge}, we remove from
  $\Gc'$ all the edges $e_1,\dots,e_{\zeta}$ obtaining a new graph
  $\Gc''$ and the new operator $\Hc_p''$ such that $f$ restricted to
  $\Gc''$ is an eigenfunction of $\Hc_p''$ with the same eigenvalue
  $\lambda$. Notice that, since we removed only nodes where $f$ is
  zero and edges where $f$ changes sign, then $\Gc''$ can be written
  as the disjoint union of the nodal domains, namely $\Gc''=\sqcup_{i=1}^{\nu(f)}\Gc_i$ and, as a consequence, we have
  \begin{equation*}
    \Hc_p''=
    \mathop{\oplus}_{i=1}^{\nu(f)}    \Hc''_{p}(\Gc_i)
  \end{equation*}
  where $ \Hc''_{p}(\Gc_i)$ is the restriction of the generalized
  $p$-Laplacian operator onto $\Gc_i$.
  Hence, from Lemma \ref{Lemma1.2} we have
  \begin{equation}\label{eq01.1}
    \lambda<\lambda_k'\leq \lambda_k''\,.
  \end{equation}
  where $\lambda_k''$ denotes the $k$-th variational eigenvalue of
  $\Hc_p''$. Note that the restriction $f|_{\Gc_i}$ to each of the
  nodal domains $\Gc_i$ of $f$ has constant sign and it is then
  necessarily the first eigenpair of $ \Hc''_{p}(\Gc_i)$ corresponding
  to $\lambda$  \new{(see Theorem \ref{Theorem0.1}) and Corollary \ref{cor:at-least-two-nd})}. Hence, $\lambda$ is also the first eigenvalue of
  $\Hc_p''$ and, as an eigenvalue of \begin{NEW} $\Hc_p''$\end{NEW}, has multiplicity exactly
  equal to $\nu(f)$. We deduce that
  $\lambda=\lambda_1''=\dots=\lambda_{\nu(f)}''$ which, combined with
  \eqref{eq01.1}, implies $k>\nu(f)$, thus concluding the proof.
\end{proof}

\begin{proof}[Proof of P1 in Theorem \ref{thm:main4-nodal-count-general}]
  Using the same notation of the proof of Theorem
  \ref{thm:main3-nodal-count-simple} above, suppose that
  $\lambda>\lambda_{k}$. Then Lemmas \ref{Lemma1.2} and \ref{Lemma2}
  imply that
  \begin{equation*}
    \lambda>\lambda_{k}\geq \lambda'_{k-z}\geq \lambda_{k-z-\zeta}''
  \end{equation*}
  where we define $\lambda_h' = \lambda_h'' = -\infty$ for $h\leq 0$
  and $\lambda_h'=\lambda_h'' = +\infty$ for $h\geq \new{N}-z\new{(f)}+1$.  As
  observed above, $\lambda$ is also the first eigenvalue of
  $\Hc_p''$. Thus, the above inequality can hold only if
$k-z\new{(f)}-\zeta\leq 0$.  Using Lemma~\ref{Lemma5} we obtain
  $k-z\new{(f)}-|E|+|E_z|-z\new{(f)}+|V|-\nu(f)+l(f)\leq 0$, with $l(f)$ being the
  number of independent loops in $\Gc$ where $f$ has constant
  sign. This implies
  \begin{equation}\label{nodaldomestim}
 \nu(f) \geq k-z(f)-(|E|-|E_z|)+(|V|-z(f))+l(f)=k+l(f)-\beta'\new{(f)}-z(f)+c\new{(f)}
  \end{equation}
  where $c\new{(f)}$ is the number of connected components of $\Gc'$ and
  $\beta'\new{(f)}:=(|E|-|E_z|)-(|V|-z\new{(f)})+c\new{(f)}$ the number of independent loops in
  $\Gc'$ .
\end{proof}

\begin{NEW}

Next, we provide a proof for \ref{it:main4-2}, Theorem~\ref{thm:main4-nodal-count-general}. The idea is similar to the one used in the proof of \ref{it:main4-1}. In the latter,  we reduced the starting graph to the disjoint union of the nodal domains of an eigenfunction and doing so we knew that the corresponding eigenvalue would become the first variational one on the reduced graph.  Now, instead, we reduce the graph to a forest where we know from Theorem \ref{thm:main1-variational-eig-tree} that our eigenvalue becomes a variational one and we know by Theorem \ref{thm:main2-nodal-count-trees} how to relate the nodal domains induced by the eigenfunction to the index of the eigenvalue.
\end{NEW}

\begin{proof}[Proof of P2 in Theorem \ref{thm:main4-nodal-count-general}]
Let $\Hc_p$ be a generalized $p$-Laplacian operator defined on a
  connected graph, $\Gc$ and assume
\new{
  \begin{equation*}
    \lambda_1\leq\lambda_2\leq\dots\leq
    \lambda_{k-1}<\lambda_{k}=\dots
    =\lambda_{k+m-1}<\lambda_{k+m}\leq\dots\leq\lambda_N\,,
  \end{equation*}
 }
  to be the variational spectrum of $\Hc_p$ and $f$ to be an
  eigenfunction relative to
  $\lambda=\lambda_{k}=\dots=\lambda_{k+\new{m}-1}$. Additionally, denote by
  $\nu(f)$ the number of nodal domains of $f$, by $l(f)$ the number of
  independent loops where $f$ has constant sign, and by
  $v_1,\dots,v_{z}$ the nodes where $f$ is zero, \new{with $z=z(f)$ the number of such nodes}.
  
  Using the results of Section~\ref{sec:remove-a-node}, we start by removing the nodes $v_1,\dots,v_{z}$ from $\Gc$ and accordingly modifying the operator $\Hc_p$, obtaining a graph $\Gc'$ and an operator $\Hc'_p$, such that the restriction of $f$ to $\Gc'$ is an
  eigenfunction of $\Hc'_p$ with the same eigenvalue $\lambda$. 
 \new{Observe that, since we have removed all and only the nodes of $\Gc$ where $f$ is zero}, $f$ restricted to $\Gc'$ has no zeros and induces the same nodal domains and constant sign loops induced on $\Gc$.  From Lemma \ref{Lemma2}, we have that
\begin{equation}\label{P2_eq1}
\lambda'_{k+\new{m}-1-z}\leq\lambda\leq\lambda_{k+\new{m}-1}'\,,
\end{equation}

  where
  $\lbrace\lambda'_k\rbrace$
  denote the variational eigenvalues of
  $\Hc'_p$.
  \begin{NEW}
  In particular $\lambda$ is an eigenvalue of $\Hc_p'$ i.e. $\lambda\in[\lambda_1',\lambda_{N-z(f)}']$ ($N-z(f)$ the number of nodes of $\Gc'$). Thus, since the variational eigenvalues of $\Hc_p'$ split its spectrum in intervals, there has to exist and index, $h$ with $\lambda'_h<\lambda'_{h+1}$ such that $\lambda\in[\lambda'_{h},\lambda'_{h+1})$ where $\lambda'_{h}=\infty$ if $h>N-z(f)$.
  Morover from \eqref{P2_eq1} we can state 
\begin{equation}\label{h_bound}
   h\geq k+m-z(f)-1\,.
\end{equation}
\end{NEW}
 Now observe that if $c(f)$ is the number of connected componets of $\Gc'$ and
  $\beta'(f)=|E'|-|V'|+c(f)$ the number of independent loops of $\Gc'$,  %
  we can remove $\beta'(f)$ edges from $\Gc'$ to obtain a forest
  $\mathcal{T}$ \begin{NEW}with the same number of connected components of $\Gc'$.\end{NEW}
  Every time we remove an edge, we modify the operator
  $\Hc_p'$ as in Section \ref{sec:remove-edge} so that the pair
  $(f,\lambda)$ remains an eigenpair of the resulting operator.  At
  each step, denote by $e_0$ the edge we are removing, and by
  $\tilde{\Gc}'$, $\tilde{\Hc}_p'$  and 
  $\tilde{\Gc}''$, $\tilde{\Hc}_p''$ the graph and the
  corresponding generalized $p$-Laplace operator before and after cutting $e_0$.  Denote by
  $\lbrace\tilde{\lambda}'_k\rbrace$ and
  $\lbrace\tilde{\lambda}''_k\rbrace$
  the variational spectra of $\tilde{\Hc}'_p$ and $\tilde{\Hc}''_p$.
  Letting $\lambda\in[\tilde{\lambda}'_{\ell},\tilde{\lambda}'_{\ell+1})$ (always with the assumption $\tilde{\lambda}'_\ell=\infty$ if $\ell>|V(\tilde{\Gc}')|$),
  due to Lemma \ref{Lemma1.2}, we can bound $\lambda$ in terms of the
  spectrum of the new operator as:
  \begin{equation*}
    \tilde{\lambda}''_{\ell-1}\leq\lambda<\tilde{\lambda}_{\ell+2}''\, .
  \end{equation*}
  \begin{NEW}
  Now, define the two counting functions $\Delta n(e_0,f)$ and
  $M(e_0,f)$. The first one counts how the  variational interval in which $\lambda$ is contained changes when moving from $\tilde {\mathcal G}'$ to $\tilde {\mathcal G}''$, namely:
  \end{NEW}
  \begin{equation*}
    \Delta n(e_0,f)= \begin{cases}
      -1 & \lambda<\tilde{\lambda}_{\ell}''\\
      +1 & \begin{NEW}\lambda\geq\tilde{\lambda}_{\ell+1}''\end{NEW}\\
      0  & \text{otherwise.}\,  
    \end{cases}
  \end{equation*}
  \begin{NEW}
 Observe that $\Delta n(e_0,f)=-1$ implies that $\lambda\in[\tilde{\lambda}_{\ell-1}'',\tilde{\lambda}_\ell'')$, $\Delta n(e_0,f)=1$ implies that $\lambda\in[\tilde{\lambda}_{\ell+1}'',\tilde{\lambda}_{\ell+2}'')$, and $\Delta n(e_0,f)=0$ implies that  $\lambda\in[\tilde{\lambda}_{\ell}'',\tilde{\lambda}_{\ell+1}'')$.

  The second counting function, $M(e_0,f)$,  takes into account the sign of $f$ on the removed edge $e_0$. Recall that,  if $\mathrm{sign}_{e_0}(f)<0$ then from Section \ref{sec:remove-edge} it follows that $\lambda\in[\tilde{\lambda}''_{\ell-1},\tilde{\lambda}''_{\ell+1})$,  otherwise we have  $\lambda\in[\tilde{\lambda}''_{\ell},\tilde{\lambda}''_{\ell+2})$. Thus, we define%
  \end{NEW}
  \begin{equation*}
    M(e_0,f):=
    \begin{cases}
      \begin{cases}
        -1 & \lambda<\tilde{\lambda}_{\ell}''\\
        0  & \text{otherwise}
      \end{cases}
      \qquad
      & \text{if} \quad  \mathrm{sign}_{e_0}(f)<0\, ,   \\[1.5em]
      \begin{cases}
        0 & \begin{NEW}\lambda\geq\tilde{\lambda}_{\ell+1}''\end{NEW}\\
        -1 & \text{otherwise}
      \end{cases}
      \qquad
      & \text{if} \quad \mathrm{sign}_{e_0}(f)>0\, .  
    \end{cases}
  \end{equation*}

   It follows by their definition and from  Lemma \ref{Lemma1.2} that 
  \begin{equation*}
    \Delta n(e_0,f)- M(e_0,f)=\begin{cases}
      0 & \mathrm{sign}_{e_0}(f)<0\\
      1 & \mathrm{sign}_{e_0}(f)>0
    \end{cases}\, .
  \end{equation*}
  Thus, thanks to \eqref{eq:Deltas}, every time we cut an edge $e_0$
  and modify consequently the operator $\Hc_p'$, we have the following
  identity
  \begin{equation}\label{variational_equality}
    \Delta n(e_0,f)- M(e_0,f)=\Delta \nu(e_0,f)- \Delta l(e_0,f)\,,
  \end{equation}
\new{ where $\Delta \nu(e_0,f)$ and $\Delta l(e_0,f)$ are  the difference between the number of
  nodal domains and the number of constant sign loops induced by $f$ in $\tilde \Gc'$  and in $\tilde \Gc''$, respectively.}
  
  After $\beta'$ steps, \new{$(f,\lambda)$} will be an eigenpair of a
  generalized $p$-Laplacian operator $\Hc_p''$ defined on the forest $\Tc$, \new{such that $f(u)\neq 0\;\forall u\in \Gc''$}.
  We have proved in Theorem \ref{thm:main1-variational-eig-tree} that
  $\Hc_p''$ has only variational eigenvalues, so, w.l.o.g., we can
  assume that $\lambda$ has became the \new{$s$}-th variational eigenvalue
  of $\Hc_p''$.  
  \begin{NEW}
  Note that, thanks to Theorem \ref{thm:main1-variational-eig-tree} and Remark \ref{remark_multiplicity_on_forests_of_everywhere_non_zero_eigenfunctions} we have that  $\mult_{\Hc_p''}(\lambda)=c(f)$, that is
$$\lambda_{s-c(f)}''<\lambda=\lambda_{s-c(f)+1}''=\dots=\lambda_{s}''<\lambda_{s+1}''.$$
 \end{NEW}
 Moreover, because of Theorem
  \ref{thm:main2-nodal-count-trees}, we know that $f$ induces
  \new{$s$} nodal domains on the forest $\mathcal{T}$.
  Thus, using \eqref{variational_equality} and the equality
  \new{$\sum_{i=1}^{\beta'}\Delta \nu(e_i,f)= s-\nu(f)$}, the number of nodal
  domains, $\nu(f)$, induced on the original graph $\Gc$ by $f$ (which is the same as the one
  induced on $\Gc'$), can be written as
  \begin{equation*}
    \nu(f)= \new{s}-\sum_{i=1}^{\beta'}\Delta \nu(e_i,f)=
    \new{s}-\sum_{i=1}^{\beta'}\Delta n(e_i,f)-\sum_{i=1}^{\beta'}\Delta
    l(e_i,f)+\sum_{i=1}^{\beta'}M(e_i,f)\,.
  \end{equation*}
  Finally, observe that\new{, by definition of $\Delta n$, it holds}
  $$\sum_{i=1}^{\beta'}\Delta n(e_i,f)=\new{s}-h,\qquad \text{and} \qquad
  \sum_{i=1}^{\beta'}\Delta l(e_i,f)=-l(f)
  $$ 
  \new{because we have removed all the loops}, while
  $\sum_{i=1}^{\beta'}M(e_i,f)\geq-\beta'(f)$ \new{(note that the equality holds if and only if  $M(e_i,f)=-1$, $\forall i$)}.  Hence, using 
  inequality \eqref{h_bound}, we obtain
  \begin{equation*}
    \nu(f)\geq \new{s}-\new{s}+h+l(f)-\beta'\new{(f)}= h+l(f)-\beta'\new{(f)}
    \geq k+\new{m}-1-z(f)+l(f)-\beta'\new{(f)}\,,
  \end{equation*}
  which concludes the proof.
\end{proof}

\appendix

\section{A technical lemma}

We devote this appendix to the following technical lemma, which is
helpful to prove several of our main results.

\begin{lemma}\label{Lemma0.1}
  Consider the function
  \begin{equation}\label{eqxi0}
    R(\beta_1,\beta_2)
    =\Big(\frac{|\beta_1|^{p}}{\phi_p\big(\alpha_1\big)}
    -\frac{|\beta_2|^{p}}{\phi_p\big(\alpha_2\big)}\Big)\phi_p
    \big(\alpha_1-\alpha_2\big)
    -\big(\beta_1-\beta_2\big)\phi_p\Big(\beta_1-\beta_2\Big)\,,
  \end{equation}
  where $\phi_p(x)=|x|^{p-2}x$, $\alpha_1,\alpha_2,\beta_1,\beta_2$
  are real numbers, and ${\alpha={\alpha_2}/{\alpha_1}}$. Then
  $R(\beta_1,\beta_2)$ is positive if $\alpha$ is negative and
  negative if $\alpha$ is positive.  Moreover $R(\beta_1,\beta_2)=0$
  if and only if ${\beta={\beta_1}/{\beta_2}={\alpha_1}/{\alpha_2}}$.
\end{lemma}
\begin{proof}
  We first consider the special cases where either $\beta_1$ or
  $\beta_2$ are zero or $\alpha=1$. When $\beta_2=0$ \eqref{eqxi0}
  becomes
  \begin{align*}
    R(\beta_1,0)=|\beta_1|^{p}\Big(\phi_p\big(1-\alpha\big)-1\Big),
  \end{align*}
  and a simple computation shows that
  $\big(\phi_p(1-\alpha)-1\big)\geq0$ if and only if $\alpha<0$.  The
  case with $\beta_1=0$ is similar, since
  $R(0,\beta_2)=|\beta_2|^{p}\Big(\phi_p\big(1-\frac{1}{\alpha}\big)-1\Big)$.
  Next, consider the case $\alpha=1$. In this case \eqref{eqxi0}
  simplifies to $R(\beta_1,\beta_2)=-|\beta_1-\beta_2|^p\leq 0$ and
  one easily sees that the equality holds if and only if
  $\beta_1=\beta_2$.

  Consider now the case where both $\beta_1$ and $\beta_2$ are
  different from zero and $\alpha\neq 1$. Equation \eqref{eqxi0} can
  be written as
  \begin{align}\label{eqxi0.1}
    R(\beta_1,\beta_2)
    =|\beta_1|^{p}\bigg(\phi_p(1-\alpha)-\phi_p
    \Big(1-\frac{\beta_2}{\beta_1}\Big)\bigg)
    +|\beta_2|^{p}\bigg(\phi_p\Big(1-\frac{1}{\alpha}\Big)
    -\phi_p\Big(1-\frac{\beta_1}{\beta_2}\Big)\bigg).
  \end{align}
  Dividing \eqref{eqxi0.1} by $|\beta_2|^p$ and letting
  $\beta = \beta_1/\beta_2$ we get the chain of inequalities
  \begin{align}
      & |\beta|^p \phi_p(1-\alpha)+\phi_p\Big(1-\frac{1}{\alpha}\Big)
        \geq |\beta|^p \phi_p\Big(1-\frac{1}{\beta} \Big)
        +\phi_p\Big(1-\beta\Big)  \notag\\
    \iff \quad &\frac{|\beta(1-\alpha)|^p}{(1-\alpha)}
                 +\frac{|1-\frac{1}{\alpha}|^p}{1-\frac{1}{\alpha}}
                 \geq |\beta-1|^{p-2}(\beta^2-\beta)
                 +|\beta-1|^{p-2}(1-\beta)  \notag\\
    \iff \quad &\frac{|\beta(1-\alpha)|^p}{(1-\alpha)}
                 +\frac{|1-\frac{1}{\alpha}|^p}{1-\frac{1}{\alpha}}
                 \geq |\beta-1|^{p} \, .\label{eq0.0.2}
  \end{align}
  Now, if $1<\alpha<0$, then ${0\!<\!\frac{1}{(1-\alpha)}\!<\!1}$ and
  ${\frac{1}{(1-\alpha)}+\frac{1}{1-\frac{1}{\alpha}}=1}$, so we can
  use the convexity of $x\mapsto |x|^p$ to obtain
  \begin{equation*}
    \frac{|\beta(1-\alpha)|^p}{(1-\alpha)}
    +\frac{|1-\frac{1}{\alpha}|^p}{1-\frac{1}{\alpha}} \geq |\beta-1|^{p}\,.
  \end{equation*}
  Since $x\mapsto |x|^p$ is strictly convex for
  $p>1$, the equality in the expression above holds if and only if
  $\beta(1-\alpha)=\frac{1}{\alpha}-1$ showing that
  $R(\beta_1,\beta_2)$ is positive if $\alpha$ is negative.

  To face the case $\alpha>0$, consider again equation
  \eqref{eqxi0.1}. We can assume without loss of generality that
  $0<\alpha<1$. Indeed, if $\alpha>1$, we can divide \eqref{eqxi0.1}
  by $|\beta_1|^p$ to obtain an equation like \eqref{eq0.0.2} where
  $1/\alpha$ is used in place of $\alpha$ and the proof would follow
  from the argument above.  Returning to the case $0<\alpha<1$, from
  \eqref{eqxi0.1} and the following sequence of inequalities can be
  obtained following the same steps as above:
  \begin{align*}
    &|\beta|^p \phi_p(1-\alpha)+\phi_p\Big(1-\frac{1}{\alpha}\Big)
      \leq |\beta-1|^p \\
    \iff \quad &|\beta|^p  \leq \frac{|\beta-1|^p}{\phi_p(1-\alpha)}
                 + \frac{\phi_p\Big(\frac{1-\alpha}{\alpha}\Big)}%
                 {\phi_p(1-\alpha)}\\
    \iff \quad &|\beta|^p  \leq \Big|\frac{\beta-1}{(1-\alpha)}
                 \Big|^p(1-\alpha) + \Big|\frac{1}{\alpha}\Big|^p\alpha
  \end{align*}
  Note that, as before, the last inequality holds due to the convexity
  of $x\mapsto |x|^p$ and thus equality holds if and only if
  $\frac{\beta-1}{(1-\alpha)}=\frac{1}{\alpha}$ which implies
  $\beta=\frac{1}{\alpha}$, concluding the proof.
\end{proof}

\section*{Acknowledgements}
We are indebted to Shiping Liu and Chuanyuan Ge for their careful reading of the first version of this work and their precious comments. We would also like to thank two anonymous referees for their very useful observations.  All their feedback greatly helped us improve the quality of the presentation of the main results as well as  some of their proofs.

\bibliographystyle{abbrvnat}
\bibliography{strings,references}

\end{document}